\documentclass[11pt]{article}

\usepackage[T1]{fontenc}
\usepackage[utf8]{inputenc}
\usepackage{amsmath,amssymb,amsthm,amsfonts,mathtools}
\usepackage{mathrsfs}
\usepackage[margin=1in]{geometry}
\usepackage{enumitem}
\usepackage{hyperref}

\hypersetup{
  colorlinks=true,
  linkcolor=blue,
  citecolor=blue,
  urlcolor=blue
}

\numberwithin{equation}{section}

\theoremstyle{plain}
\newtheorem{theorem}{Theorem}[section]
\newtheorem{proposition}[theorem]{Proposition}
\newtheorem{lemma}[theorem]{Lemma}
\newtheorem{corollary}[theorem]{Corollary}

\theoremstyle{definition}
\newtheorem{definition}[theorem]{Definition}

\theoremstyle{remark}
\newtheorem{remark}[theorem]{Remark}

\newcommand{\Vol}{\operatorname{Vol}}
\newcommand{\Tr}{\operatorname{Tr}}

\newcommand{\PSL}{\mathrm{PSL}}
\newcommand{\SL}{\mathrm{SL}}
\newcommand{\SO}{\mathrm{SO}}

\newcommand{\HH}{\mathbb H}

\newcommand{\R}{\mathbb R}
\newcommand{\Z}{\mathbb Z}

\newcommand{\Hor}{\mathcal H}
\newcommand{\Lcal}{\mathcal L}
\newcommand{\Mcal}{\mathcal M}
\newcommand{\Rcal}{\mathcal R}

\newcommand{\Fcal}{\mathcal F}

\newcommand{\ii}{\mathrm i}
\newcommand{\ee}{\mathrm e}

\newcommand{\detF}{\det\nolimits_{\!F}}
\renewcommand{\Re}{\operatorname{Re}}

\DeclareMathOperator{\arccosh}{arccosh}
\DeclareMathOperator{\Arg}{Arg}

\title{Sub-Riemannian Selberg Trace Formulae for Compact Quotients of $\SL(2,\mathbb R)$ and Determinants of Sub-Laplacians}

\author{Fabrice Baudoin\thanks{Research partially supported by grant 10.46540/4283-00175B from Independent Research Fund Denmark and by the Villum Investigator grant \emph{Stochastic Analysis in Aarhus}. I also acknowledge funding from the European Research Council (ERC) under the European Union's Horizon Europe research and innovation programme (RanGe project, Grant Agreement No. 101199772).}}
\date{\today}

\begin{document}

\maketitle

\begin{abstract}
We prove sub-Riemannian Selberg trace formulae for compact quotients of
\(\SL(2,\mathbb R)\).
Using the Fourier decomposition along the \(\SO(2)\)-fibers, we reduce
the heat trace computation to the Selberg trace formula for Maass
Laplacians on the hyperbolic plane.  The resulting formula has an
identity contribution and a hyperbolic contribution, the latter involving
a theta factor that records both the twisting character and the central
sign of the chosen lift to \(\SL(2,\mathbb R)\).  We then use
this trace formula to compute the zeta-regularized determinant of the
sub-Laplacian.  The determinant formula is expressed as a universal factor
depending only on the base hyperbolic surface times an explicit
lift- and character-dependent Selberg-type product.
\end{abstract}

\tableofcontents

\section{Introduction}
\label{sec:intro}

The Selberg trace formula is a fundamental bridge between spectral geometry
and the dynamics of closed geodesics.  For a compact hyperbolic surface
\(M=\bar\Gamma\backslash\HH\), it relates the spectrum of the
Laplace--Beltrami operator to the length spectrum of \(M\), or equivalently
to the hyperbolic conjugacy classes of \(\bar\Gamma\).  In its heat-kernel
form, the trace formula expresses the heat trace as the sum of an identity
contribution and a contribution over closed geodesics.  This form is
particularly well adapted to determinant questions, since zeta-regularized
determinants are obtained from Mellin transforms of reduced heat traces.
We refer to \cite{Hejhal,McKean1972,McKean1974,Patterson} for the classical
theory and for the relation with the Selberg zeta function. The motivation for the present work is to understand to what extent the
classical spectral geometry of compact hyperbolic surfaces admits a natural
sub-Riemannian analogue on compact quotients of \(\SL(2,\mathbb R)\). The sub-Riemannian structure arising from the horizontal distribution of the fibration
\(\SL(2,\mathbb R)\to\mathbb H\) was  thoroughly studied by Bonnefont
\cite{Bonnefont} (see also \cite{Wang}) and compact quotients of \(\SL(2,\mathbb R)\) carry natural hypoelliptic
sub-Laplacians.  However, comparatively little is known about explicit
Selberg-type trace formulae and determinant formulae for these hypoelliptic
operators.  The aim of this paper is therefore to develop a parallel picture
in this sub-Riemannian setting: to derive a heat trace formula for the
sub-Laplacian, to identify the closed-geodesic contribution, and to compute
the corresponding zeta-regularized determinant in terms of base-surface
spectral data and an explicit lift- and character-dependent Selberg-type
product.

More precisely, we consider
\[
        G=\SL(2,\R),\qquad K=\SO(2).
\]
Let \(X,Y,Z\) be the basis of the Lie algebra \(\mathfrak{sl}(2,\R)\)
satisfying
\[
        [X,Y]=Z,\qquad [X,Z]=Y,\qquad [Y,Z]=-X.
\]
Denoting the corresponding left-invariant vector fields by the same letters,
we consider the horizontal left-invariant distribution
\[
        \Hor=\operatorname{span}\{X,Y\}.
\]
The intrinsic sub-Laplacian is
\[
        \Lcal=X^2+Y^2.
\]
Unlike the Laplace--Beltrami operator on a Riemannian manifold, \(\Lcal\)
is not elliptic.  It is nevertheless hypoelliptic by H\"ormander's
bracket-generating condition, and on compact quotients it has discrete
spectrum.

Let \(\Gamma<G\) be a cocompact discrete subgroup whose projection
\(\bar\Gamma\subset\PSL(2,\R)\) is torsion-free and contains no elliptic
elements, and assume that \(\Gamma\cap K=\{I\}\).  Then
\[
        M=\bar\Gamma\backslash\HH
\]
is a compact hyperbolic surface.  For a character
\[
        \chi:\Gamma\to \R/2\pi\Z,
\]
we let \(\Gamma\) act on \(G\) by
\[
        \gamma\cdot_\chi A=\gamma A R(\chi(\gamma)),
\]
where \(R(\chi(\gamma))\) denotes the rotation of angle \(\chi(\gamma)\).
We denote the corresponding compact quotient by
\[
        P^{\Gamma,\chi}:=\Gamma\backslash_\chi G.
\]
Since \(\Lcal\) is left-invariant and right \(K\)-invariant, it descends to
an operator \(\Lcal_{\Gamma,\chi}\) on \(P^{\Gamma,\chi}\).  Our first main
result is a Selberg trace formula for the heat operator $\ee^{t\Lcal_{\Gamma,\chi}} $.
More precisely, if
\[
        0=\lambda_0^\chi\leq \lambda_1^\chi\leq\lambda_2^\chi\leq\cdots
\]
are the eigenvalues of \(-\Lcal_{\Gamma,\chi}\), repeated with multiplicity,
then Theorem~\ref{thm:twisted-sr-selberg-sl2} gives the sub-Riemannian
Selberg trace formula
\[
        \sum_{j=0}^\infty \ee^{-t\lambda_j^\chi}
        =
        \mathcal I_\chi(t)+\mathcal H_\chi(t),
        \qquad t>0.
\]
The identity contribution is independent of \(\chi\) and is given by
\[
\mathcal I_\chi(t)
=
\frac{\pi\Vol(M)\ee^{-t/4}}{4t^2}
\sum_{k\in\Z}
\frac{
\exp\!\left(-\frac{4k^2\pi^2}{t}\right)
}{
\cosh^2\!\left(\frac{2|k|\pi^2}{t}\right)
},
\]
whereas the hyperbolic contribution is
\[
\mathcal H_\chi(t)
=
\frac{\ee^{-t/4}}{4\sqrt{\pi t}}
\sum_{[\gamma]\in\mathcal P_{\bar\Gamma}}
\sum_{r=1}^{\infty}
\Theta_{\Gamma,\chi}(t;\gamma^r)
\frac{\ell(\gamma)}
{\sinh\!\left(\frac{r\ell(\gamma)}2\right)}
\exp\!\left(-\frac{r^2\ell(\gamma)^2}{4t}\right).
\]
Here \(\mathcal P_{\bar\Gamma}\) is the set of primitive hyperbolic conjugacy
classes, with \([\gamma]\) and \([\gamma^{-1}]\) counted separately,
\(\ell(\gamma)\) is the length of the corresponding closed geodesic,
and, if \(\widetilde\gamma\in\Gamma\) is the lift of a primitive
representative, 
\[
        \Theta_{\Gamma,\chi}(t;\gamma^r)
        =
        \sum_{m\in\Z}
        \ee^{-m^2t/4}
        \ee^{-\ii mr(\chi(\widetilde\gamma)+\kappa_\Gamma(\widetilde\gamma))}
\]
is the fiberwise theta factor.  Here
\(\kappa_\Gamma(\widetilde\gamma)\in\{0,\pi\}\) records whether
\(\widetilde\gamma\) is conjugate in \(\SL(2,\R)\) to the positive or
negative lift of the standard hyperbolic diagonal element.  Thus the trace formula has the same
identity/hyperbolic structure as the classical Selberg heat trace formula in
\cite{McKean1972}, but the closed-geodesic contribution is weighted by an
additional fiberwise theta series.

The proof uses the Fourier decomposition along the \(K=\SO(2)\)-fibers.  In
Iwasawa coordinates, the \(m\)-th Fourier mode of the sub-Laplacian is
identified with a Maass Laplacian of weight \(m\).  This reduces the orbital
integrals for the sub-Riemannian heat kernel to the corresponding orbital
integrals in the Selberg trace formula for Maass heat kernels.  The required
heat-kernel input on \(\SL(2,\R)\), which yields the identity contribution,
is taken from Bonnefont's heat kernel formula \cite{Bonnefont}, while the hyperbolic
orbital integrals are computed using the Maass heat trace formula of
Ikeda--Matsumoto \cite{IkedaMatsumoto}.  Since the base surface is compact
and the group is torsion-free, there are no parabolic or elliptic
contributions.

The second main result of the paper is the computation of the
zeta-regularized determinant of the positive subelliptic operator
\(-\Lcal_{\Gamma,\chi}\).  Zeta-regularized determinants of Laplacians are
spectral invariants which package the full non-zero spectrum into a single
geometric quantity.  They arise naturally from the Mellin transform of the
reduced heat trace.  In Riemannian geometry, such determinants play a
central role in the study of analytic torsion, following the work of Ray and
Singer \cite{RaySinger1971}, and in the spectral geometry of surfaces, where
they are closely connected with the Selberg zeta function and the length
spectrum \cite{Sarnak1987,Hejhal,Patterson}.  They also vary non-trivially
with the metric and have been used as refined spectral functionals on moduli
spaces of surfaces \cite{OsgoodPhillipsSarnak1988}.  In the sub-Riemannian
and contact settings, zeta functions and determinants of sub-Laplacians have
been studied explicitly on several model spaces and nilmanifolds
\cite{BauerFurutani,BauerFurutaniIwasaki,BauerFurutaniIwasakiPseudoH}, but
explicit determinant formulae of the type considered here remain scarce.

After removing the simple zero eigenvalue, we define the zeta function of
\(-\Lcal_{\Gamma,\chi}\) by
\[
        \zeta_{\Gamma,\chi}(s)
        =
        \sum_{j=1}^\infty (\lambda_j^\chi)^{-s},
        \qquad \Re s>2.
\]
The heat trace formula gives the meromorphic continuation of
\(\zeta_{\Gamma,\chi}\) to a neighborhood of \(s=0\), where it is regular,
and we set
\[
        \det\nolimits_\zeta' (-\Lcal_{\Gamma,\chi})
        :=
        \exp\bigl(-\zeta_{\Gamma,\chi}'(0)\bigr).
\]
Our main determinant result is the formula
\[
        \det\nolimits_\zeta' (-\Lcal_{\Gamma,\chi})
        =
        (4\pi)^2
        \exp\bigl(C_{\mathrm{id}}\Vol(M)\bigr)
        \exp\!\left(4\pi\zeta_M\!\left(-\frac12\right)\right)
        \detF\!\left(I-\ee^{-4\pi\sqrt{\Delta_M'}}\right)^2
        \Rcal_{\Gamma,\chi} .
\]
Here \(C_{\mathrm{id}}\) is a universal constant, \(\Delta_M'\) denotes the
restriction of the Laplace--Beltrami operator of the base surface \(M\) to
the orthogonal complement of the constants, \(\zeta_M\) is the reduced
spectral zeta function of \(\Delta_M\), \(\detF\) is the Fredholm
determinant, and \(\Rcal_{\Gamma,\chi}\) is an explicit Selberg-type
product over closed geodesics that depends on the lift and the twisting
character through their total fiber holonomy. Our determinant formula is closely related in spirit to the classical formulas of D'Hoker and Phong \cite{DHP} for determinants of Laplacians on tensors and spinors over compact hyperbolic Riemann surfaces. In their work, fixed-weight Maass Laplacians give rise to determinants expressed in terms of Selberg zeta values at shifted half-integer points, up to explicit universal factors depending only on the topology and the weight. The present situation is different in that the operator is a hypoelliptic sub-Laplacian on a compact quotient of \( SL(2,\mathbb R) \). The Fourier decomposition along the \(SO(2)\)-fibers produces an infinite family of Maass-type contributions, and the determinant is obtained by regularizing this full tower of contributions. Thus the final formula may in a sense be viewed as a sub-Riemannian analogue of the D'Hoker--Phong determinant formulas, with the lift and character dependence encoded in an explicit closed-geodesic product.

The paper is organized as follows.  In Section~\ref{sec:sl2-structure} we
fix the geometric normalization on \(\SL(2,\R)\), record the relevant
Iwasawa-coordinate identities needed to apply the results of \cite{Bonnefont}, and prove the sub-Riemannian Selberg trace
formula on the twisted compact quotients \(P^{\Gamma,\chi}\).  In
Section~\ref{sec:sub-laplacian-determinant} we use this trace formula to
continue the spectral zeta function, derive the finite-part formula for the
determinant, compute the universal base-surface factor, and isolate the lift
and character dependence as a closed-geodesic product. Finally, in Section \ref{subsec:universal-identity-constant} we give a formula for the universal constant $C_{\mathrm{id}}$.

\section{The sub-Riemannian Selberg trace formula on compact quotients of  \texorpdfstring{$\SL(2,\R)$}{SL(2,R)}}
\label{sec:sl2-structure}

Let
\[
        G=\SL(2,\R),\qquad K=\SO(2).
\]
We use the basis
\[
X=\frac12
\begin{pmatrix}
1&0\\
0&-1
\end{pmatrix},
\qquad
Y=\frac12
\begin{pmatrix}
0&1\\
1&0
\end{pmatrix},
\qquad
Z=\frac12
\begin{pmatrix}
0&1\\
-1&0
\end{pmatrix}
\]
of the Lie algebra $\mathfrak{sl}(2,\R)$.  Then
\begin{equation}\label{eq:lie-relations}
[X,Y]=Z,\qquad [X,Z]=Y,\qquad [Y,Z]=-X.
\end{equation}
The corresponding left-invariant vector fields will be denoted by the same letters.  We endow $G$ with the left-invariant Riemannian metric for which $X,Y,Z$ is an orthonormal frame.  The horizontal distribution is
\[
        \Hor=\operatorname{span}\{X,Y\},
\]
and the intrinsic horizontal sub-Laplacian is
\begin{equation}\label{eq:sub-laplacian-sl2}
        \Lcal=X^2+Y^2.
\end{equation}

Let
\[
R(\theta)=
\begin{pmatrix}
\cos\theta&-\sin\theta\\
\sin\theta&\cos\theta
\end{pmatrix}\in K,
\qquad 0\leq\theta<2\pi,
\]
and, for $z=y_1+\ii y_2$, $y_1 \in \mathbb R, y_2>0$, set
\[
s(z)=
\begin{pmatrix}
\sqrt{y_2}&y_1/\sqrt{y_2}\\
0&1/\sqrt{y_2}
\end{pmatrix}.
\]
The map
\[
        (z,\theta)\longmapsto s(z)R(\theta)
\]
gives global Iwasawa coordinates on $G$.  We can identify $G/K$ with the Poincar\'e upper half-plane
\[
        \HH=\{z=y_1+\ii y_2:y_1\in\R,\ y_2>0\}
\]
with metric $(dy_1^2+dy_2^2)/y_2^2$.  In the coordinates $(y_1,y_2,\theta)$ one can compute that
\[
X
=
 y_2\sin(2\theta)\frac{\partial}{\partial y_1}
+
 y_2\cos(2\theta)\frac{\partial}{\partial y_2}
+
\frac12\sin(2\theta)\frac{\partial}{\partial \theta},
\]
\[
Y
=
 y_2\cos(2\theta)\frac{\partial}{\partial y_1}
-
 y_2\sin(2\theta)\frac{\partial}{\partial y_2}
+
\frac12\cos(2\theta)\frac{\partial}{\partial \theta},
\]
and
\[
        Z=-\frac12\frac{\partial}{\partial\theta}.
\]
Consequently
\begin{equation}\label{eq:Lcal-coordinates}
\Lcal
=
 y_2^2
\left(
\frac{\partial^2}{\partial y_1^2}
+
\frac{\partial^2}{\partial y_2^2}
\right)
+
 y_2\frac{\partial^2}{\partial y_1\partial\theta}
+
\frac14\frac{\partial^2}{\partial\theta^2}.
\end{equation}
A Haar measure we consider and fix as our reference measure throughout the paper is
\begin{equation}\label{eq:sl2-haar}
       d\mu= \frac{dy_1\,dy_2\,d\theta}{y_2^2}.
\end{equation}
This differs by a factor $1/2$ from the Riemannian volume.
\subsection{Preliminary lemmas}
\label{subsec:iwasawa-identities}

For
\[
A=
\begin{pmatrix}
a&b\\
c&d
\end{pmatrix}\in\SL(2,\R),
\]
we write the induced action on $\HH$ as
\[
        Az=\frac{az+b}{cz+d},
\]
and set
\[
        \alpha_A(z)=\arg(cz+d).
\]
We shall use repeatedly the elementary central-sign identity
\begin{equation}\label{eq:alpha-central-sign}
        \alpha_{-A}(z)=\alpha_A(z)+\pi
        \quad\text{in }\R/2\pi\Z.
\end{equation}

\begin{lemma}\label{lem:iwasawa-action}
For every $A\in\SL(2,\R)$ and $z\in\HH$,
\[
        A\,s(z)=s(Az)R(\alpha_A(z)).
\]
\end{lemma}

\begin{proof}
Both sides send $i\in\HH$ to $Az$.  Hence $A s(z)=s(Az)k$ for a unique $k\in K$.  Write $k=R(\alpha)$, $z=y_1+\ii y_2$, and $Az=w_1+\ii w_2$.  Since
\[
        w_2=\frac{y_2}{|cz+d|^2},
\]
comparison of the lower rows gives
\[
(c\sqrt{y_2},(cy_1+d)/\sqrt{y_2})
=
\frac1{\sqrt{w_2}}(\sin\alpha,\cos\alpha).
\]
Thus
\[
        \sin\alpha=\frac{c y_2}{|cz+d|},
        \qquad
        \cos\alpha=\frac{c y_1+d}{|cz+d|},
\]
which is precisely $\alpha=\arg(cz+d)$ modulo $2\pi$.
\end{proof}

The next elementary identity relates the polar and Iwasawa fiber coordinates.  It will be used to compare heat-kernel formulae written in the two coordinate systems.

\begin{lemma}\label{lem:polar-iwasawa-fiber}
Suppose
\[
        M=P R(\zeta),
\]
where $P$ is symmetric positive definite.  If also
\[
        M=s(z)R(\theta),\qquad z=y_1+\ii y_2,
\]
then
\[
        \zeta=\theta+\Arg\bigl((y_2+1)-\ii y_1\bigr)\quad\mathrm{mod}\ 2\pi.
\]
Equivalently,
\[
        \zeta=\theta-\arctan\left(\frac{y_1}{y_2+1}\right)
        \quad\mathrm{mod}\ 2\pi.
\]
\end{lemma}

\begin{proof}
The angle $\zeta$ is determined by the condition that $P=M R(-\zeta)$ be symmetric.  Write
\[
M=
\begin{pmatrix}
\alpha&\beta\\
\gamma&\delta
\end{pmatrix}.
\]
Since
\[
R(-\zeta)=
\begin{pmatrix}
\cos\zeta&\sin\zeta\\
-\sin\zeta&\cos\zeta
\end{pmatrix},
\]
the symmetry condition is
\[
        (\alpha+\delta)\sin\zeta+(\beta-\gamma)\cos\zeta=0.
\]
For $M=s(z)R(\theta)$ one computes
\[
\alpha+\delta=
\frac{(y_2+1)\cos\theta+y_1\sin\theta}{\sqrt{y_2}},
\qquad
\beta-\gamma=
\frac{y_1\cos\theta-(y_2+1)\sin\theta}{\sqrt{y_2}}.
\]
Hence
\[
\tan\zeta=
\frac{(y_2+1)\sin\theta-y_1\cos\theta}
{(y_2+1)\cos\theta+y_1\sin\theta}.
\]
Equivalently,
\[
\begin{aligned}
\zeta
&=\Arg\!\bigl((y_2+1)\cos\theta+y_1\sin\theta
+\ii((y_2+1)\sin\theta-y_1\cos\theta)\bigr)\\
&=\theta+\Arg\bigl((y_2+1)-\ii y_1\bigr)
\quad\mathrm{mod}\ 2\pi.
\end{aligned}
\]
Since $y_2+1>0$, this gives the stated arctangent form.
\end{proof}

\subsection{The heat kernel and the Maass decomposition}
\label{sec:sl2-heat}

Let $p_t(g,h)$ be the heat kernel of $\ee^{t\Lcal}$ on $G$ with respect to the Haar measure \eqref{eq:sl2-haar}.  We write $g=(z,\theta)=s(z)R(\theta)$.

\begin{proposition}[Subelliptic heat kernel on $\SL(2,\R)$]
\label{prop:heat-kernel-sl2}
For every $t>0$, the diagonal value of the heat kernel is independent of $(z,\theta)$ and is given by
\begin{equation}\label{eq:diagonal-heat-kernel-sl2}
p_t((z,\theta),(z,\theta))
=
\frac{\ee^{-t/4}}{8t^2}
\sum_{k\in\Z}
\frac{
\exp\!\left(-\frac{4k^2\pi^2}{t}\right)
}{
\cosh^2\!\left(\frac{2|k|\pi^2}{t}\right)
}.
\end{equation}
For $u,z\in\HH$, $u\neq z$, and $\phi,\theta\in[0,2\pi)$,
\begin{align}
& p_t((u,\phi),(z,\theta))                                      \notag\\
&=
\frac{\ee^{-t/4}}{8\pi^{5/2}t^{3/2}}
\sum_{m\in\Z}
\ee^{-tm^2/4}
\ee^{\ii m(\theta-\phi)}
\left(
\ii\,\frac{\bar u-z}{|\bar u-z|}
\right)^m                                                       \notag\\
&\quad\times
\int_{-\infty}^{\infty}
\ee^{-A(\rho,\eta)^2/t}
\ee^{-m\eta}
\frac{A(\rho,\eta)}
{\sqrt{\cosh^2\rho\,\cosh^2\eta-1}}
\,d\eta,
\label{eq:offdiag-heat-kernel-sl2}
\end{align}
where
\[
        \rho=\frac{d_{\HH}(u,z)}2,
        \qquad
        A(\rho,\eta)=\arccosh(\cosh\rho\,\cosh\eta).
\]
\end{proposition}

\begin{proof}
By left-invariance of $\mathcal L$, the diagonal heat kernel is constant; in particular
\[
p_t((z,\theta),(z,\theta))=p_t((\mathrm i,0),(\mathrm i,0)).
\]
We first record the normalization needed to use Bonnefont's formula.
Bonnefont's Lie-algebra basis is twice the basis used here, so that
\(\Lcal_{\mathrm B}=4\Lcal\).  With the Haar normalization in
\cite{Bonnefont}, a direct comparison in exponential coordinates gives
\(d\mu=4\,d\mu_{\mathrm B}\).  Consequently the two heat kernels satisfy
\begin{equation}\label{eq:bonnefont-normalization}
        p_t(g,h)=\frac14p^{\mathrm B}_{t/4}(g,h).
\end{equation}
After this time and density rescaling, Bonnefont's formula
\cite[Proposition 3.7]{Bonnefont} gives
\begin{align*}
p_t((\mathrm i,0),(\mathrm i,0))
&=\frac{e^{-t/4}}{2t^2}\sum_{k \in \mathbb{Z}} \exp\left(-\frac{4k^2\pi^2}{t} \right)
\frac{\exp \left( -\frac{4|k| \pi^2}{t}\right)}{\left( 1+\exp \left( -\frac{4|k| \pi^2}{t}\right)\right)^2} \\
&=\frac{\ee^{-t/4}}{8t^2}
\sum_{k\in\Z}
\frac{
\exp\!\left(-\frac{4k^2\pi^2}{t}\right)
}{
\cosh^2\!\left(\frac{2|k|\pi^2}{t}\right)
}.
\end{align*}
For the off-diagonal formula, left-invariance and \cite[Proposition 3.3]{Bonnefont} yield
\begin{align*}
    p_t((u,\phi),(z,\theta)) &=p_t((i,0),(u,\phi)^{-1}(z,\theta)) \\
    & =\frac{e^{-t/4}}{(2\pi t)^2} \sum_{k \in \mathbb Z} \int_{-\infty}^{\infty}
\ee^{-A(\rho,\eta)^2/t} \ee^{\frac{(\eta+\ii \zeta-2 \ii k\pi)^2 }{t}}
\frac{A(\rho,\eta)}
{\sqrt{\cosh^2\rho\,\cosh^2\eta-1}}
\,d\eta,
\end{align*}
where Lemmas \ref{lem:iwasawa-action} and \ref{lem:polar-iwasawa-fiber} give
\[
\zeta=\theta-\phi+  \arg (\ii(\bar{u}-z)).
\]
Interchanging the sum and the integral and using Poisson summation in the form
\[
\sum_{k\in\mathbb Z}
\exp\left(\frac{(\eta+i\zeta-2\pi i k)^2}{t}\right)
=
\sqrt{\frac{t}{4\pi}}
\sum_{m\in\mathbb Z}
\exp\left(-\frac{t m^2}{4}+m\eta+im\zeta\right)
\]
gives \eqref{eq:offdiag-heat-kernel-sl2} after the change of variable $\eta \to -\eta$ in the integral.
\end{proof}

We now relate the subelliptic heat kernel to the Maass Laplacians. Such a
connection was first pointed out in \cite{BaudoinDemni}, but non-trivial
changes of variables are required to use those formulas in our setting, so
we proceed with a direct computation.

For $B\in\R$, define the  operator
\begin{equation}\label{eq:maass-operator-sl2}
\Mcal_B
=
-\frac12 y_2^2
\left(
\frac{\partial^2}{\partial y_1^2}
+
\frac{\partial^2}{\partial y_2^2}
\right)
+
\ii B y_2\frac{\partial}{\partial y_1}
+
\frac{B^2}{2}.
\end{equation}
Equivalently,
\[
D_B=-2\Mcal_B+B^2
=
 y_2^2
\left(
\frac{\partial^2}{\partial y_1^2}
+
\frac{\partial^2}{\partial y_2^2}
\right)
-2\ii B y_2\frac{\partial}{\partial y_1}
\]
is the usual Maass Laplacian of weight $2B$, see \cite{IkedaMatsumoto}.

\begin{lemma}\label{lem:Lcal-maass-relation-sl2}
For every $m\in\Z$ and every smooth function $f$ on $\HH$,
\[
        \Lcal(\ee^{-\ii m\theta}f)
        =
        -2\ee^{-\ii m\theta}\Mcal_{m/2}f.
\]
\end{lemma}

\begin{proof}
Using \eqref{eq:Lcal-coordinates},
\[
\frac{\partial}{\partial\theta}(\ee^{-\ii m\theta}f)=-\ii m\ee^{-\ii m\theta}f,
\qquad
\frac{\partial^2}{\partial\theta^2}(\ee^{-\ii m\theta}f)=-m^2\ee^{-\ii m\theta}f.
\]
Substitution gives
\[
\Lcal(\ee^{-\ii m\theta}f)
=
\ee^{-\ii m\theta}
\left[
 y_2^2\left(\frac{\partial^2 f}{\partial y_1^2}+\frac{\partial^2f }{\partial y_2^2}\right)
-\ii m y_2 \frac{\partial f}{\partial y_1}
-\frac{m^2}{4}f
\right],
\]
which is $-2\ee^{-\ii m\theta}\Mcal_{m/2}f$.
\end{proof}

Let $q_B(\tau;u,z)$ denote the heat kernel of $\ee^{-\tau\Mcal_B}$ on $\HH$ with respect to
\[
        dm_{\HH}(z)=\frac{dy_1\,dy_2}{y_2^2}.
\]
Lemma~\ref{lem:Lcal-maass-relation-sl2} gives the Fourier expansion
\begin{equation}\label{eq:subelliptic-maass-decomposition-sl2}
p_t((u,\phi),(z,\theta))
=
\frac1{2\pi}
\sum_{m\in\Z}
\ee^{\ii m(\theta-\phi)}
q_{m/2}(2t;u,z).
\end{equation}
By uniqueness of Fourier coefficients, Proposition~\ref{prop:heat-kernel-sl2} therefore implies the following explicit formula.

\begin{corollary}\label{cor:maass-heat-kernel-explicit}
For every $m\in\Z$, $t>0$, and $u,z\in\HH$, $u\neq z$,
\begin{align*}
q_{m/2}(2t;u,z)=
\frac{\ee^{-t(m^2+1)/4}}{4\pi^{3/2}t^{3/2}}  \left(
\ii\,\frac{\bar u-z}{|\bar u-z|}
\right)^m
\int_{-\infty}^{\infty}
\ee^{-A(\rho,\eta)^2/t}
\ee^{-m\eta}
\frac{A(\rho,\eta)}
{\sqrt{\cosh^2\rho\,\cosh^2\eta-1}}
\,d\eta.
\end{align*}
\end{corollary}

\subsection{Twisted compact quotients and the sub-Riemannian Selberg trace formula}
\label{sec:sl2-quotients}

Let $\Gamma<G$ be a cocompact discrete subgroup whose projection
\[
        \bar\Gamma\subset\PSL(2,\R)
\]
is torsion-free and contains no elliptic elements.  We assume throughout that
\[
        \Gamma\cap K=\{I\}.
\]
Thus $M=\bar\Gamma\backslash\HH$ is a compact hyperbolic surface, and the projection $\Gamma\to\bar\Gamma$ is an isomorphism.  We shall therefore identify $\Gamma$ and $\bar\Gamma$ whenever no confusion is possible.
In formulas involving the fiber holonomy, however, we retain the unique
lift in \(\Gamma\), because its trace can be positive or negative and this
central sign is detected by the odd fiber Fourier modes.

Given a character $\chi:\Gamma\to\R/2\pi\Z$, define the twisted left action of $\Gamma$ on $G$ by
\begin{equation}\label{eq:twisted-action}
        \gamma\cdot_\chi A=\gamma A R(\chi(\gamma)).
\end{equation}
By Lemma~\ref{lem:iwasawa-action},
\begin{equation}\label{eq:twisted-action-iwasawa}
        \gamma\cdot_\chi s(z)R(\theta)
        =s(\gamma z)R\bigl(\theta+\alpha_\gamma(z)+\chi(\gamma)\bigr).
\end{equation}
We write
\[
        P^{\Gamma,\chi}:=\Gamma\backslash_\chi G.
\]
With the Haar measure \eqref{eq:sl2-haar},
\[
        \Vol(P^{\Gamma,\chi})=2\pi\Vol(M).
\]
If $\Fcal\subset\HH$ is a fundamental domain for the action of $\bar\Gamma$ on $\HH$, then
\[
        \mathscr F
        =
        \{s(z)R(\theta):z\in\Fcal,\ 0\leq\theta<2\pi\}
\]
is a fundamental domain for the action of $\Gamma$ on $G$ defined by \eqref{eq:twisted-action}.

Since $\Lcal$ commutes with left translations and is invariant under the right $K$-action, it descends to an operator $\Lcal_{\Gamma,\chi}$ on $P^{\Gamma,\chi}$.  The operator $-\Lcal_{\Gamma,\chi}$ is non-negative, essentially self-adjoint, subelliptic, and has discrete spectrum.

The heat kernel on $P^{\Gamma,\chi}$ is obtained by periodization:
\begin{equation}\label{eq:quotient-heat-kernel-sl2}
p_t^{\Gamma,\chi}(g,h)
=
\sum_{\gamma\in\Gamma}p_t(g,\gamma\cdot_\chi h).
\end{equation}
Consequently
\begin{equation}\label{eq:trace-periodized-sl2}
\Tr(\ee^{t\Lcal_{\Gamma,\chi}})
=
\int_{\mathscr F}p_t^{\Gamma,\chi}(g,g)\,dg
=
\sum_{\gamma\in\Gamma}
\int_{\mathscr F}p_t(g,\gamma\cdot_\chi g)\,dg.
\end{equation}

Let $\mathcal P_{\bar\Gamma}$ denote the set of primitive hyperbolic conjugacy classes in $\bar\Gamma$.
We count $[\gamma]$ and $[\gamma^{-1}]$ separately throughout the paper.
These classes are distinct: an orientation-preserving isometry of \(\HH\)
conjugating $\gamma$ to $\gamma^{-1}$ would reverse its axis and have
order two, contrary to the torsion-free hypothesis.
If $[\gamma]\in\mathcal P_{\bar\Gamma}$, let $\ell(\gamma)>0$ be its translation length on $\HH$.  Equivalently, if $\gamma\in\bar\Gamma$ is conjugate in $\PSL(2,\R)$ to
\[
        \begin{pmatrix} a&0\\0&a^{-1}\end{pmatrix},
        \qquad a>1,
\]
then $\ell(\gamma)=2\log a$.

For every non-identity hyperbolic element $\delta\in\Gamma$, define its
central phase by
\begin{equation}\label{eq:central-phase}
\kappa_\Gamma(\delta)
=
\begin{cases}
0,&\Tr(\delta)>0,\\
\pi,&\Tr(\delta)<0.
\end{cases}
\end{equation}
Thus, if \(\ell(\delta)\) denotes the translation length of its projection,
then \(\delta\) is conjugate in \(\SL(2,\R)\) to
\[
\ee^{\ii\kappa_\Gamma(\delta)}
\begin{pmatrix}
\ee^{\ell(\delta)/2}&0\\
0&\ee^{-\ell(\delta)/2}
\end{pmatrix}.
\]
Although \(\kappa_\Gamma\) need not be a group character, it is constant on
conjugacy classes and satisfies
\begin{equation}\label{eq:central-phase-powers}
        \kappa_\Gamma(\delta^r)
        =r\kappa_\Gamma(\delta)
        \quad\text{in }\R/2\pi\Z,
        \qquad r\geq1.
\end{equation}
The total fiber phase is
\begin{equation}\label{eq:total-fiber-phase}
        \phi_{\Gamma,\chi}(\delta)
        :=\chi(\delta)+\kappa_\Gamma(\delta)
        \quad\text{in }\R/2\pi\Z.
\end{equation}
Define the corresponding theta factor by
\begin{equation}\label{eq:theta-character-sl2}
\Theta_{\Gamma,\chi}(t;\delta)
=
\sum_{m\in\Z}
\ee^{-m^2t/4}\ee^{-\ii m\phi_{\Gamma,\chi}(\delta)}.
\end{equation}
Equivalently,
\begin{equation}\label{eq:theta-character-cosine-sl2}
\Theta_{\Gamma,\chi}(t;\delta)
=
1+2\sum_{m=1}^{\infty}
\ee^{-m^2t/4}\cos\bigl(m\phi_{\Gamma,\chi}(\delta)\bigr).
\end{equation}
Since $\chi$ is a character and \eqref{eq:central-phase-powers} holds,
$\Theta_{\Gamma,\chi}(t;\delta)$ is constant on conjugacy classes and
\[
        \Theta_{\Gamma,\chi}(t;\delta^r)
        =
        \sum_{m\in\Z}
        \ee^{-m^2t/4}\ee^{-\ii mr\phi_{\Gamma,\chi}(\delta)}.
\]

\begin{theorem}
\label{thm:twisted-sr-selberg-sl2}
Let $\chi:\Gamma\to\R/2\pi\Z$ be a character, and let
\[
        0=\lambda_0^\chi\leq\lambda_1^\chi\leq\lambda_2^\chi\leq\cdots
\]
be the eigenvalues of $-\Lcal_{\Gamma,\chi}$, repeated with multiplicity.  Then, for every $t>0$,
\begin{equation}\label{eq:main-trace-formula-character-sl2}
\sum_{j=0}^{\infty}\ee^{-t\lambda_j^\chi}
=
\mathcal I_\chi(t)+\mathcal H_\chi(t),
\end{equation}
where the identity contribution is independent of $\chi$ and is given by
\begin{equation}\label{eq:identity-contribution-character-sl2}
\mathcal I_\chi(t)
=
\frac{\pi\Vol(M)\ee^{-t/4}}{4t^2}
\sum_{k\in\Z}
\frac{
\exp\!\left(-\frac{4k^2\pi^2}{t}\right)
}{
\cosh^2\!\left(\frac{2|k|\pi^2}{t}\right)
},
\end{equation}
and the hyperbolic contribution is
\begin{equation}\label{eq:hyperbolic-contribution-character-sl2}
\mathcal H_\chi(t)
=
\frac{\ee^{-t/4}}{4\sqrt{\pi t}}
\sum_{[\gamma]\in\mathcal P_{\bar\Gamma}}
\sum_{r=1}^{\infty}
\Theta_{\Gamma,\chi}(t;\gamma^r)
\frac{\ell(\gamma)}
{\sinh\!\left(\frac{r\ell(\gamma)}2\right)}
\exp\!\left(-\frac{r^2\ell(\gamma)^2}{4t}\right).
\end{equation}
In $\Theta_{\Gamma,\chi}(t;\gamma^r)$ we use the unique lift to
$\Gamma$ of a representative of the primitive class $[\gamma]$; the
value is independent of the representative.
\end{theorem}

\begin{proof}
Fix \(t>0\).  The heat operator \(\ee^{t\Lcal_{\Gamma,\chi}}\) is
positive and trace class, since \(-\Lcal_{\Gamma,\chi}\) is a
non-negative subelliptic operator on the compact manifold
\(P^{\Gamma,\chi}\).

The right \(K\)-action descends to \(P^{\Gamma,\chi}\), preserves its
measure, and commutes with \(\Lcal_{\Gamma,\chi}\).  For \(m\in\Z\), let
\[
\mathscr E_m
:=
\left\{
F\in L^2(P^{\Gamma,\chi}):
F(gR(\beta))=\ee^{\ii m\beta}F(g)
\text{ for every }\beta\in\R/2\pi\Z
\right\}.
\]
Then
\[
        L^2(P^{\Gamma,\chi})
        =\widehat{\bigoplus}_{m\in\Z}\mathscr E_m.
\]
Write
\[
        T_{m,\chi}(t)
        :=\Tr_{\mathscr E_m}(\ee^{t\Lcal_{\Gamma,\chi}})
        \geq0,
\]
and let \(\Pi_N\) be the orthogonal projection onto
\(\bigoplus_{|m|\leq N}\mathscr E_m\).  Since \(\Pi_N\) commutes with
the heat operator and increases strongly to the identity,
\begin{equation}\label{eq:finite-fiber-mode-trace}
\Tr(\Pi_N\ee^{t\Lcal_{\Gamma,\chi}})
=
\sum_{|m|\leq N}T_{m,\chi}(t)
\ \uparrow\
\Tr(\ee^{t\Lcal_{\Gamma,\chi}}).
\end{equation}
We compute each fixed-mode trace and then pass to this limit.

A smooth function in \(\mathscr E_m\), lifted to \(G\), has the form
\[
        F(s(z)R(\theta))=\ee^{\ii m\theta}f(z).
\]
Invariance under the twisted action is equivalent to
\begin{equation}\label{eq:fixed-mode-automorphy}
        f(\delta z)
        =
        \ee^{-\ii m(\alpha_\delta(z)+\chi(\delta))}f(z),
        \qquad \delta\in\Gamma.
\end{equation}
Put \(B_m=-m/2\).  Lemma~\ref{lem:Lcal-maass-relation-sl2}, with
\(m\) replaced by \(-m\), gives
\[
        -\Lcal(\ee^{\ii m\theta}f)
        =
        2\ee^{\ii m\theta}\Mcal_{B_m}f.
\]
The map \(f\mapsto(2\pi)^{-1/2}\ee^{\ii m\theta}f\) is unitary from
the corresponding space of automorphic sections on \(M\) onto
\(\mathscr E_m\).  Thus \(T_{m,\chi}(t)\) is the trace of
\(\ee^{-2t\Mcal_{B_m}}\) with automorphy law
\eqref{eq:fixed-mode-automorphy}.

To specify its multiplier, let
\[
        \widehat\Gamma:=\Gamma\cup(-\Gamma),
\]
the full inverse image of \(\bar\Gamma\) in \(\SL(2,\R)\).
Since \(-I\notin\Gamma\), each element of \(\widehat\Gamma\) is uniquely
of the form \((-I)^\varepsilon\delta\), where
\(\varepsilon\in\{0,1\}\) and \(\delta\in\Gamma\).  Define the unitary
character
\[
        \rho_{m,\chi}((-I)^\varepsilon\delta)
        :=(-1)^{m\varepsilon}\ee^{-\ii m\chi(\delta)}.
\]
Then \(\rho_{m,\chi}(-I)=(-1)^m\), and
\eqref{eq:fixed-mode-automorphy} is equivalent to
\[
        f(Az)=\rho_{m,\chi}(A)\ee^{-\ii m\alpha_A(z)}f(z),
        \qquad A\in\widehat\Gamma.
\]
Indeed, the two factors contributed by \(-I\) cancel.
This is the central compatibility condition for the fixed-weight
trace formula; see also \cite[Sections 4.5--4.6]{ChoiTakhtajan}.

Let \(\widetilde\gamma\in\Gamma\) lift a primitive representative
\(\gamma\in\bar\Gamma\), and let \(\gamma_+\in\widehat\Gamma\) be the
positive-trace lift of that same element.  Since
\[
        \widetilde\gamma
        =\ee^{\ii\kappa_\Gamma(\widetilde\gamma)}\gamma_+,
\]
the multiplier on its \(r\)-th positive-trace power is
\begin{equation}\label{eq:maass-induced-multiplier}
\begin{aligned}
        v_{m,\Gamma,\chi}(\gamma^r)
        &:=
        \rho_{m,\chi}(\gamma_+^r)\\
        &=
        \ee^{-\ii mr(\chi(\widetilde\gamma)
        +\kappa_\Gamma(\widetilde\gamma))}
        =
        \ee^{-\ii mr\phi_{\Gamma,\chi}(\widetilde\gamma)}.
\end{aligned}
\end{equation}
Here the sign of \(\kappa_\Gamma\) in the exponential is immaterial
because \(\kappa_\Gamma\in\{0,\pi\}\).

Apply the Maass heat trace formula
\cite[Theorem 6.1]{IkedaMatsumoto} with this multiplier, heat-time
\(\tau=2t\), and \(B=B_m=-m/2\).  Since the effective group
\(\bar\Gamma\) is cocompact and torsion-free, it has no parabolic or
non-identity elliptic contributions.  The formula gives
\begin{equation}\label{eq:fixed-mode-heat-trace}
        T_{m,\chi}(t)=\mathcal I_m(t)+\mathcal H_{m,\chi}(t),
\end{equation}
where
\[
        \mathcal I_m(t)
        =
        \Vol(M)\,q_{-m/2}(2t;\ii,\ii)
\]
and
\begin{align}
\mathcal H_{m,\chi}(t)
&=
\frac{\ee^{-t/4}}{4\sqrt{\pi t}}\,\ee^{-m^2t/4}
\sum_{[\gamma]\in\mathcal P_{\bar\Gamma}}
\sum_{r=1}^{\infty}                                      \notag\\
&\quad\times
\ee^{-\ii mr\phi_{\Gamma,\chi}(\widetilde\gamma)}
\frac{\ell(\gamma)}
{\sinh\!\left(\frac{r\ell(\gamma)}2\right)}
\exp\!\left(-\frac{r^2\ell(\gamma)^2}{4t}\right).
\label{eq:fixed-mode-hyperbolic-term}
\end{align}
In particular, the time and weight factors are
\[
        \frac{\ee^{-\tau/8}}{2\sqrt{2\pi\tau}}\,
        \ee^{-B_m^2\tau/2}
        =
        \frac{\ee^{-t/4}}{4\sqrt{\pi t}}\,\ee^{-m^2t/4}.
\]
All applications of the trace formula up to this point concern a
single fixed mode.

We now sum \eqref{eq:fixed-mode-heat-trace} over \(|m|\leq N\).
The diagonal values \(q_{-m/2}(2t;\ii,\ii)\) are non-negative by the
semigroup property and self-adjointness.  The diagonal Fourier expansion
\eqref{eq:subelliptic-maass-decomposition-sl2} therefore gives
\[
\begin{aligned}
        \lim_{N\to\infty}\sum_{|m|\leq N}\mathcal I_m(t)
        &=
        2\pi\Vol(M)\,p_t(e,e)\\
        &=
        \frac{\pi\Vol(M)\ee^{-t/4}}{4t^2}
        \sum_{k\in\Z}
        \frac{\exp\!\left(-\frac{4k^2\pi^2}{t}\right)}
        {\cosh^2\!\left(\frac{2|k|\pi^2}{t}\right)},
\end{aligned}
\]
where the last equality is Proposition~\ref{prop:heat-kernel-sl2}.
This is \(\mathcal I_\chi(t)\).

Let \(\mathcal H_M(t)\) denote the positive double series
\eqref{eq:fixed-mode-hyperbolic-term} with \(m=0\); it is independent
of \(\chi\).  The positive systole and the at most exponential growth
of the number of primitive classes imply that
\(\mathcal H_M(t)<\infty\), since its summands have Gaussian decay in
\(r\ell(\gamma)\).  As the multipliers have modulus one,
\[
        |\mathcal H_{m,\chi}(t)|
        \leq \ee^{-m^2t/4}\mathcal H_M(t).
\]
More precisely, the sum of the absolute values of all summands in
\eqref{eq:fixed-mode-hyperbolic-term}, also summed over \(m\in\Z\),
is
\[
        \left(\sum_{m\in\Z}\ee^{-m^2t/4}\right)\mathcal H_M(t)
        <\infty.
\]
Consequently the hyperbolic terms converge absolutely as
\(N\to\infty\), and they may be regrouped by primitive classes and
positive powers.  Their sum is
\[
\begin{aligned}
\sum_{m\in\Z}\mathcal H_{m,\chi}(t)
&=
\frac{\ee^{-t/4}}{4\sqrt{\pi t}}
\sum_{[\gamma]\in\mathcal P_{\bar\Gamma}}
\sum_{r=1}^{\infty}\\
&\quad\times
\left(
\sum_{m\in\Z}
\ee^{-m^2t/4}\ee^{-\ii mr\phi_{\Gamma,\chi}(\widetilde\gamma)}
\right)
\frac{\ell(\gamma)}
{\sinh\!\left(\frac{r\ell(\gamma)}2\right)}
\exp\!\left(-\frac{r^2\ell(\gamma)^2}{4t}\right)\\
&=\mathcal H_\chi(t).
\end{aligned}
\]
Taking the limit in the finite sum of
\eqref{eq:fixed-mode-heat-trace} and using
\eqref{eq:finite-fiber-mode-trace} proves
\[
        \Tr(\ee^{t\Lcal_{\Gamma,\chi}})
        =\mathcal I_\chi(t)+\mathcal H_\chi(t).
\]
The spectral theorem gives the asserted sum over eigenvalues.
Finally, \(\chi\) and \(\kappa_\Gamma\) are invariant under conjugation,
and \eqref{eq:central-phase-powers} gives their required power law.
Thus the theta factor is independent of the primitive representative.
\end{proof}

\begin{remark}
\label{rem:sr-trace-vs-surface-heat-trace}
Let
\[
        Z_{\mathrm{sR},\chi}(t)=\Tr(\ee^{t\Lcal_{\Gamma,\chi}}),
        \qquad
        Z_M(t)=\Tr(\ee^{-t\Delta_M}),
\]
where $\Delta_M$ is the positive Laplace--Beltrami operator on $M$.  The classical Selberg heat trace formula  \cite{McKean1972,McKean1974} has the form
\[
        Z_M(t)=\mathcal I_M(t)+\mathcal H_M(t),
\]
where
\begin{equation}\label{eq:base-identity-term}
\mathcal I_M(t)
=
\frac{\Vol(M)\ee^{-t/4}}{4\pi}
\int_{-\infty}^{\infty}
\ee^{-t\lambda^2}\lambda\tanh(\pi\lambda)\,d\lambda
\end{equation}
and
\[
\mathcal H_M(t)
=
\frac{\ee^{-t/4}}{4\sqrt{\pi t}}
\sum_{[\gamma]\in\mathcal P_{\bar\Gamma}}
\sum_{r=1}^{\infty}
\frac{\ell(\gamma)}
{\sinh\!\left(\frac{r\ell(\gamma)}2\right)}
\exp\!\left(-\frac{r^2\ell(\gamma)^2}{4t}\right).
\]
Set
\[
        \vartheta(t):=\sum_{m\in\Z}\ee^{-m^2t/4}
\]
and introduce the closed-geodesic correction
\begin{equation}\label{eq:geodesic-correction-heat}
        \mathcal G_{\Gamma,\chi}(t)
        :=\mathcal H_\chi(t)-\vartheta(t)\mathcal H_M(t).
\end{equation}
The second term is the formal zero-holonomy reference contribution.  The
actual lift contributes the phase \(\kappa_\Gamma\), even when \(\chi=0\).
Consequently the correct comparison is
\begin{equation}\label{eq:untwisted-trace-comparison}
        Z_{\mathrm{sR},\chi}(t)
        =
        \mathcal I_0(t)
        +\vartheta(t)\bigl(Z_M(t)-\mathcal I_M(t)\bigr)
        +\mathcal G_{\Gamma,\chi}(t).
\end{equation}
This comparison separates the determinant into a universal base-surface
factor and a lift- and character-dependent closed-geodesic factor.  In
particular, one cannot in general replace
\(\mathcal H_0(t)\) by \(\vartheta(t)\mathcal H_M(t)\).
\end{remark}
\section{The zeta determinant of the sub-Laplacian}
\label{sec:sub-laplacian-determinant}

We now use the trace formula to define and compute the zeta-regularized
determinant of the positive subelliptic operator
\(-\Lcal_{\Gamma,\chi}\).  We denote its heat trace by
\[
        \mathscr T_\chi(t)
        :=
        \Tr(\ee^{t\Lcal_{\Gamma,\chi}}).
\]
By Theorem~\ref{thm:twisted-sr-selberg-sl2},
\begin{equation}\label{eq:det-heat-trace-decomposition}
        \mathscr T_\chi(t)=\mathcal I_\chi(t)+\mathcal H_\chi(t),
        \qquad t>0.
\end{equation}

\subsection{The spectral zeta function}
\label{subsec:spectral-zeta-sub-laplacian}

\begin{lemma}\label{lem:kernel-sub-laplacian}
Let $\chi:\Gamma\to \R/2\pi\Z$ be a character.  Then
\[
        \ker(-\Lcal_{\Gamma,\chi})
        =
        \{\text{constant functions}\}.
\]
In particular, the eigenvalue \(0\) of \(-\Lcal_{\Gamma,\chi}\) is simple.
\end{lemma}

\begin{proof}
The twisted action preserves the horizontal distribution and its metric.
Let \(\nabla_{\Hor}\) denote the resulting horizontal gradient on
\(P^{\Gamma,\chi}\).  Although the vector fields \(X,Y\) need not
descend individually, their squared derivatives define the intrinsic
energy identity
\[
        \langle -\Lcal f,f\rangle_{L^2(P^{\Gamma,\chi})}
        =
        \int_{P^{\Gamma,\chi}}|\nabla_{\Hor}f|^2\,d\mu,
        \qquad f\in C^\infty(P^{\Gamma,\chi}).
\]
Indeed, on every local lift to \(G\), the energy density is
\(|X\widetilde f|^2+|Y\widetilde f|^2\), which is invariant under
the orthogonal changes of horizontal frame induced by the twisted action.

If \(f\in\ker(-\Lcal_{\Gamma,\chi})\), hypoellipticity implies that
\(f\) is smooth.  The energy identity gives \(\nabla_{\Hor}f=0\).
Its lift \(\widetilde f\) to \(G\) therefore satisfies
\(X\widetilde f=Y\widetilde f=0\), and hence
\[
        Z\widetilde f=[X,Y]\widetilde f=0.
\]
Since \(X,Y,Z\) span the tangent bundle of the connected group \(G\),
the function \(\widetilde f\), and therefore \(f\), is constant.
Conversely, constants are annihilated by \(\Lcal_{\Gamma,\chi}\).
\end{proof}

Removing the zero eigenvalue, we write the reduced heat trace as
\[
        \mathscr T_\chi^*(t):=\mathscr T_\chi(t)-1.
\]

The identity term in the trace formula gives the leading short-time terms explicitly. Indeed, the
\(k=0\) term in \eqref{eq:identity-contribution-character-sl2} is
\[
        \frac{\pi\Vol(M)}{4t^2}\ee^{-t/4},
\]
whereas the terms with \(k\neq0\) are exponentially small as
\(t\downarrow0\).  The hyperbolic contribution is also exponentially
small as \(t\downarrow0\), since the length spectrum has a positive
minimum.  Therefore
\begin{equation}\label{eq:reduced-heat-short-time}
\mathscr T_\chi^*(t)
=
\frac{\pi\Vol(M)}4 t^{-2}
-
\frac{\pi\Vol(M)}{16}t^{-1}
+
\left(\frac{\pi\Vol(M)}{128}-1\right)
+O(t).
\end{equation}
Let
\[
        0=\lambda_0^\chi<\lambda_1^\chi\leq\lambda_2^\chi\leq\cdots
\]
be the spectrum of \(-\Lcal_{\Gamma,\chi}\), repeated with multiplicity.  For
\(\Re s>2\), define
\begin{equation}\label{eq:zeta-sub-laplacian-definition}
        \zeta_{\Gamma,\chi}(s)
        :=
        \sum_{j=1}^{\infty}(\lambda_j^\chi)^{-s}.
\end{equation}
Equivalently,
\begin{equation}\label{eq:zeta-mellin-sub-laplacian}
        \zeta_{\Gamma,\chi}(s)
        =
        \frac1{\Gamma(s)}
        \int_0^\infty
        t^{s-1}\mathscr T_\chi^*(t)\,dt,
        \qquad \Re s>2.
\end{equation}
The convergence at infinity follows from the spectral gap above the
simple zero eigenvalue.  The convergence at zero follows from the
short-time expansion.

Set
\begin{equation}\label{eq:det-heat-coefficients}
        c_{-2}:=\frac{\pi\Vol(M)}4,
        \qquad
        c_{-1}:=-\frac{\pi\Vol(M)}{16},
        \qquad
        c_0:=\frac{\pi\Vol(M)}{128}-1.
\end{equation}

\begin{theorem}
\label{thm:zeta-continuation-sub-laplacian}
The function \(\zeta_{\Gamma,\chi}(s)\) extends meromorphically to a
neighborhood of \(s=0\), is regular at \(s=0\), and satisfies
\begin{equation}\label{eq:zeta-zero-value-sub-laplacian}
        \zeta_{\Gamma,\chi}(0)=\frac{\pi\Vol(M)}{128}-1.
\end{equation}
Moreover, near \(s=0\),
\begin{align}
\zeta_{\Gamma,\chi}(s)
&=
\frac1{\Gamma(s)}
\bigg[
\int_0^1
 t^{s-1}
\Bigl(
\mathscr T_\chi^*(t)-c_{-2}t^{-2}-c_{-1}t^{-1}-c_0
\Bigr)\,dt                                      \notag\\
&\qquad
+
\frac{c_{-2}}{s-2}
+
\frac{c_{-1}}{s-1}
+
\frac{c_0}{s}
+
\int_1^\infty t^{s-1}\mathscr T_\chi^*(t)\,dt
\bigg].
\label{eq:zeta-regularized-near-zero}
\end{align}
\end{theorem}

\begin{proof}
Split the Mellin integral \eqref{eq:zeta-mellin-sub-laplacian} at
\(t=1\).  On the interval \((0,1)\), subtract and add the three leading terms in
\eqref{eq:reduced-heat-short-time}.  Since
\[
\mathscr T_\chi^*(t)-c_{-2}t^{-2}-c_{-1}t^{-1}-c_0=O(t),
\qquad t\downarrow0,
\]
the integral
\[
\int_0^1
 t^{s-1}
\Bigl(
\mathscr T_\chi^*(t)-c_{-2}t^{-2}-c_{-1}t^{-1}-c_0
\Bigr)\,dt
\]
is holomorphic for \(\Re s>-1\).  The subtracted terms contribute
\[
        \frac{c_{-2}}{s-2}
        +
        \frac{c_{-1}}{s-1}
        +
        \frac{c_0}{s}.
\]
On \((1,\infty)\), Lemma~\ref{lem:kernel-sub-laplacian} gives a positive
spectral gap above the zero eigenvalue, so
\[
        \mathscr T_\chi^*(t)=O(\ee^{-\lambda_1^\chi t}),
        \qquad t\to\infty.
\]
Hence
\[
        \int_1^\infty t^{s-1}\mathscr T_\chi^*(t)\,dt
\]
is entire in \(s\).  This proves \eqref{eq:zeta-regularized-near-zero}.

Finally, denoting by $\gamma_{\mathrm E}$ the Euler--Mascheroni constant,
\[
        \frac1{\Gamma(s)}=s+\gamma_{\mathrm E}s^2+O(s^3),
        \qquad s\to0,
\]
and the bracketed expression in \eqref{eq:zeta-regularized-near-zero}
has the form
\[
        \frac{c_0}{s}+O(1).
\]
Thus \(\zeta_{\Gamma,\chi}\) is regular at \(s=0\), and
\[
        \zeta_{\Gamma,\chi}(0)=c_0
        =
        \frac{\pi\Vol(M)}{128}-1.
\]
\end{proof}
Since the function \(\zeta_{\Gamma,\chi}(s)\) is regular at \(s=0\) we set the following definition.

\begin{definition}\label{def:zeta-determinant-sub-laplacian}
The reduced zeta-regularized determinant of \(-\Lcal_{\Gamma,\chi}\) is
\begin{equation}\label{eq:zeta-det-sub-laplacian}
        \det\nolimits_\zeta'(-\Lcal_{\Gamma,\chi})
        :=
        \exp\bigl(-\zeta_{\Gamma,\chi}'(0)\bigr).
\end{equation}
\end{definition}

Our goal is to give a formula for the reduced zeta-regularized determinant. We start with the following lemma.

\begin{lemma}
\label{cor:finite-part-determinant}
One has
\begin{align}
\log\det\nolimits_\zeta'(-\Lcal_{\Gamma,\chi})
&=
-
\int_0^1
\frac{
\mathscr T_\chi^*(t)-c_{-2}t^{-2}-c_{-1}t^{-1}-c_0
}{t}\,dt                                      \notag\\
&\quad
-
\int_1^\infty
\frac{\mathscr T_\chi^*(t)}{t}\,dt
+
\frac{c_{-2}}2+c_{-1}-\gamma_{\mathrm E} c_0,
\label{eq:det-finite-part-formula}
\end{align}
where \(\gamma_{\mathrm E}\) is Euler's constant.
\end{lemma}

\begin{proof}
Let \(B(s)\) denote the bracketed expression in
\eqref{eq:zeta-regularized-near-zero}.  From the proof of
Theorem~\ref{thm:zeta-continuation-sub-laplacian},
\[
        B(s)=\frac{c_0}{s}+B_0+O(s),
        \qquad s\to0,
\]
where
\[
\begin{aligned}
B_0
&=
\int_0^1
\frac{
\mathscr T_\chi^*(t)-c_{-2}t^{-2}-c_{-1}t^{-1}-c_0
}{t}\,dt
+
\int_1^\infty
\frac{\mathscr T_\chi^*(t)}{t}\,dt  \\
&\qquad
-
\frac{c_{-2}}2-c_{-1}.
\end{aligned}
\]
Since \(1/\Gamma(s)=s+\gamma_{\mathrm E}s^2+O(s^3)\), we get
\[
        \zeta_{\Gamma,\chi}(s)
        =
        c_0+s\bigl(B_0+\gamma_{\mathrm E}c_0\bigr)+O(s^2).
\]
Therefore
\[
        \log\det\nolimits_\zeta'(-\Lcal_{\Gamma,\chi})
        =
        -\zeta_{\Gamma,\chi}'(0)
        =
        -B_0-\gamma_{\mathrm E}c_0,
\]
which gives \eqref{eq:det-finite-part-formula}.
\end{proof}

To compute $\det\nolimits_\zeta'(-\Lcal_{\Gamma,\chi})$, we separate the
universal base-surface contribution from the lift- and
character-dependent closed-geodesic contribution.

\subsection{The universal base factor}
\label{subsec:untwisted-fredholm-factor}

Let
\[
        Z_M(t):=\Tr(\ee^{-t\Delta_M})
\]
be the heat trace of the positive Laplace--Beltrami operator on the
compact hyperbolic surface \(M=\bar\Gamma\backslash\HH\).  Write
\[
        0=\mu_0<\mu_1\leq\mu_2\leq\cdots
\]
for the spectrum of \(\Delta_M\), repeated with multiplicity, and let
\(\Delta_M'\) denote the restriction of \(\Delta_M\) to the orthogonal
complement of the constants.  We also write
\[
        \zeta_M(s)
        :=
        \Tr{}((\Delta'_M)^{-s})
        =
        \sum_{j=1}^{\infty}\mu_j^{-s},
\]
initially defined for \(\Re s>1\), and then by meromorphic continuation, see \cite[Section 4.2]{McKean1972}.

Set
\[
        \vartheta(t):=\sum_{m\in\Z}\ee^{-m^2t/4}.
\]
By Remark~\ref{rem:sr-trace-vs-surface-heat-trace},
\[
        \mathscr T_\chi(t)
        =
        \mathcal I_0(t)
        +\vartheta(t)\bigl(Z_M(t)-\mathcal I_M(t)\bigr)
        +\mathcal G_{\Gamma,\chi}(t).
\]
After subtracting the common zero eigenvalue, this becomes
\begin{equation}\label{eq:untwisted-reduced-split}
        \mathscr T_\chi^*(t)
        =
        \vartheta(t)\bigl(Z_M(t)-1\bigr)
        +
        \bigl(\vartheta(t)-1\bigr)
        +
        \bigl(\mathcal I_0(t)-\vartheta(t)\mathcal I_M(t)\bigr)
        +\mathcal G_{\Gamma,\chi}(t).
\end{equation}

We first compute the contributions of the first three terms in
\eqref{eq:untwisted-reduced-split}.  They form the universal base factor:
the first term produces the Fredholm determinant,
the second produces $(4\pi)^2$, and the third produces
$\exp(C_{\mathrm{id}}\Vol(M))$.  The final term will produce the
lift- and character-dependent closed-geodesic product.

The first term is the theta-weighted base term
\begin{equation}\label{eq:theta-weighted-base-zeta}
        \zeta_{\vartheta,M}(s)
        :=
        \frac1{\Gamma(s)}
        \int_0^\infty
        t^{s-1}
        \vartheta(t)\bigl(Z_M(t)-1\bigr)\,dt.
\end{equation}
For \(\Re s\) large enough, we have
\[
        \zeta_{\vartheta,M}(s)
        =
        \sum_{j=1}^{\infty}
        \sum_{m\in\Z}
        \left(\mu_j+\frac{m^2}{4}\right)^{-s}.
\]
The following proposition evaluates its determinant contribution directly.

\begin{proposition}
\label{prop:theta-weighted-fredholm}
The function \(\zeta_{\vartheta,M}(s)\) is regular at \(s=0\), and
\begin{equation}\label{eq:theta-weighted-fredholm-factor}
        \exp\bigl(-\zeta_{\vartheta,M}'(0)\bigr)
        =
        \exp\!\left(4\pi\zeta_M\!\left(-\frac12\right)\right)
        \detF\!\left(I-\ee^{-4\pi\sqrt{\Delta_M'}}\right)^2 ,
\end{equation}
where $\detF$ denotes the Fredholm determinant.
\end{proposition}

\begin{proof}
We first notice that from  Poisson summation formula
\begin{equation}\label{eq:theta-poisson}
        \vartheta(t)
        =
        2\sqrt{\frac{\pi}{t}}
        \sum_{q\in\Z}
        \exp\!\left(-\frac{4\pi^2q^2}{t}\right).
\end{equation}
Substituting this into \eqref{eq:theta-weighted-base-zeta} gives
\[
\zeta_{\vartheta,M}(s)
=
\zeta_{\vartheta,M}^{(0)}(s)
+
\zeta_{\vartheta,M}^{(\neq0)}(s),
\]
where the superscript \(0\) denotes the \(q=0\) Poisson term.

For the \(q=0\) term, we obtain
\[
\begin{aligned}
\zeta_{\vartheta,M}^{(0)}(s)
&=
2\sqrt\pi\,
\frac1{\Gamma(s)}
\int_0^\infty
t^{s-\frac32}
\bigl(Z_M(t)-1\bigr)\,dt  \\
&=
2\sqrt\pi\,
\frac{\Gamma\left(s-\frac12\right)}{\Gamma(s)}
\zeta_M\!\left(s-\frac12\right),
\end{aligned}
\]
again first in its domain of absolute convergence and then by
meromorphic continuation.  Since
\[
        \frac1{\Gamma(s)}=s+O(s^2),
        \qquad
        \Gamma\left(s-\frac12\right)
        =
        \Gamma\left(-\frac12\right)+O(s)
        =
        -2\sqrt\pi+O(s),
\]
we have
\[
        \zeta_{\vartheta,M}^{(0)}(s)
        =
        -4\pi s\,\zeta_M\!\left(-\frac12\right)
        +
        O(s^2).
\]
Hence
\begin{equation}\label{eq:qzero-theta-derivative}
        \left(\zeta_{\vartheta,M}^{(0)}\right)'(0)
        =
        -4\pi\zeta_M\!\left(-\frac12\right).
\end{equation}

It remains to compute the \(q\neq0\) terms.  Using the spectral expansion
of \(Z_M(t)-1\), we write
\[
\zeta_{\vartheta,M}^{(\neq0)}(s)
=
\frac{2\sqrt\pi}{\Gamma(s)}
\sum_{q\in\Z\setminus\{0\}}
\sum_{j=1}^{\infty}
\int_0^\infty
t^{s-\frac32}
\exp\!\left(
-\mu_jt-\frac{4\pi^2q^2}{t}
\right)\,dt .
\]
For \(q\neq0\), the integral is exponentially regular at \(t=0\), and
therefore the expression multiplying \(1/\Gamma(s)\) is holomorphic at
\(s=0\).  Since \(1/\Gamma(s)=s+O(s^2)\), differentiation at zero gives
\[
\left(\zeta_{\vartheta,M}^{(\neq0)}\right)'(0)
=
2\sqrt\pi
\sum_{q\in\Z\setminus\{0\}}
\sum_{j=1}^{\infty}
\int_0^\infty
t^{-\frac32}
\exp\!\left(
-\mu_jt-\frac{4\pi^2q^2}{t}
\right)\,dt .
\]
Using the elementary integral
\[
        \int_0^\infty
        t^{-\frac32}
        \exp\!\left(-\beta t-\frac{\gamma}{t}\right)\,dt
        =
        \frac{\sqrt\pi}{\sqrt\gamma}
        \exp(-2\sqrt{\beta\gamma}),
        \qquad \beta,\gamma>0,
\]
with \(\beta=\mu_j\) and \(\gamma=4\pi^2q^2\), gives
\[
        \int_0^\infty
        t^{-\frac32}
        \exp\!\left(
        -\mu_jt-\frac{4\pi^2q^2}{t}
        \right)\,dt
        =
        \frac{1}{2\sqrt\pi\,|q|}
        \exp(-4\pi |q|\sqrt{\mu_j}).
\]
Consequently,
\begin{equation}\label{eq:qnonzero-theta-derivative}
\left(\zeta_{\vartheta,M}^{(\neq0)}\right)'(0)
=
2\sum_{q=1}^{\infty}
\frac1q
\sum_{j=1}^{\infty}
\exp(-4\pi q\sqrt{\mu_j}).
\end{equation}

Now set
\[
        Q:=\ee^{-4\pi\sqrt{\Delta_M'}}.
\]
The operator \(Q\) is trace class, and
\[
        \Tr(Q^q)
        =
        \sum_{j=1}^{\infty}\exp(-4\pi q\sqrt{\mu_j}).
\]
Therefore \eqref{eq:qnonzero-theta-derivative} becomes
\[
\left(\zeta_{\vartheta,M}^{(\neq0)}\right)'(0)
=
2\sum_{q=1}^{\infty}\frac{\Tr(Q^q)}{q}.
\]
By the standard Fredholm determinant identity
\[
        \log\detF(I-Q)
        =
        -\sum_{q=1}^{\infty}\frac{\Tr(Q^q)}{q},
\]
we get
\[
        -\left(\zeta_{\vartheta,M}^{(\neq0)}\right)'(0)
        =
        2\log\detF(I-Q).
\]
Combining this with \eqref{eq:qzero-theta-derivative} gives
\[
\begin{aligned}
        -\zeta_{\vartheta,M}'(0)
        &=
        4\pi\zeta_M\!\left(-\frac12\right)
        +
        2\log\detF\!\left(I-\ee^{-4\pi\sqrt{\Delta_M'}}\right),
\end{aligned}
\]
which is exactly \eqref{eq:theta-weighted-fredholm-factor}.
\end{proof}

The second term in \eqref{eq:untwisted-reduced-split} is 
\[
        \vartheta(t)-1
        =
        \sum_{m\in\Z\setminus\{0\}}\ee^{-m^2t/4}.
\]
Its zeta function is the classical zeta function on the circle
\begin{equation}\label{eq:pure-vertical-zeta}
        \zeta_{\mathbb S}(s)
        =
        \sum_{m\in\Z\setminus\{0\}}
        \left(\frac{m^2}{4}\right)^{-s}
        =
        2\, \cdot 4^s\zeta_{\mathrm R}(2s),
\end{equation}
where $\zeta_{\mathrm R}$ is the Riemann zeta function.
Using
\[
        \zeta_{\mathrm R}(0)=-\frac12,
        \qquad
        \zeta_{\mathrm R}'(0)=-\frac12\log(2\pi),
\]
we find
\begin{equation}\label{eq:pure-vertical-det}
        \exp\bigl(-\zeta_{\mathbb S}'(0)\bigr)
        =
        (4\pi)^2.
\end{equation}

Finally, define
\begin{equation}\label{eq:a-id-definition}
        a_{\mathrm{id}}(t)
        :=
        \frac{\mathcal I_0(t)}{\Vol(M)}
        -
        \vartheta(t)\frac{\mathcal I_M(t)}{\Vol(M)}.
\end{equation}
Its Mellin transform is
\begin{equation}\label{eq:zeta-id-definition}
        \zeta_{\mathrm{id}}(s)
        :=
        \frac1{\Gamma(s)}
        \int_0^\infty t^{s-1}a_{\mathrm{id}}(t)\,dt,
\end{equation}
continued meromorphically to a neighborhood of \(s=0\), and we set
\begin{equation}\label{eq:C-id-definition}
        C_{\mathrm{id}}
        :=
        -\zeta_{\mathrm{id}}'(0).
\end{equation}
The function \(a_{\mathrm{id}}\) only comes from the identity
contributions and does not depend on \(\Gamma\).  Thus
\(C_{\mathrm{id}}\) is a universal numerical constant; see
Proposition~\ref{prop:C-id-explicit} below.

Define the universal base-surface factor
\begin{equation}\label{eq:base-reference-determinant}
\begin{aligned}
D_{\mathrm{base}}(M)
&:=(4\pi)^2
\exp\bigl(C_{\mathrm{id}}\Vol(M)\bigr)
\exp\!\left(4\pi\zeta_M\!\left(-\frac12\right)\right)\\
&\qquad\times
\detF\!\left(I-\ee^{-4\pi\sqrt{\Delta_M'}}\right)^2.
\end{aligned}
\end{equation}
The calculations above show that this is the determinant contribution of
the first three terms in \eqref{eq:untwisted-reduced-split}.

\subsection{Lift and character dependence}
\label{subsec:relative-character-determinant}

We now compute the contribution of
\(\mathcal G_{\Gamma,\chi}\).  For a primitive class \([\gamma]\), let
\(\widetilde\gamma\in\Gamma\) denote its unique lift and abbreviate
\begin{equation}\label{eq:primitive-total-phase}
        \phi_{\Gamma,\chi}(\gamma)
        :=\kappa_\Gamma(\widetilde\gamma)
        +\chi(\widetilde\gamma)
        \quad\text{in }\R/2\pi\Z.
\end{equation}
For \(m\geq1\), define the holomorphic phase-decorated Selberg-type product
\begin{equation}\label{eq:character-selberg-product}
Z_{\Gamma,\chi}^{(m)}(s)
:=
\prod_{[\gamma]\in\mathcal P_{\bar\Gamma}}
\prod_{n=0}^{\infty}
\left(
1-\ee^{\ii m\phi_{\Gamma,\chi}(\gamma)}
\ee^{-(s+n)\ell(\gamma)}
\right),
\qquad \Re s>1.
\end{equation}
We also write \(Z_0(s)\) for the same product with phase zero.  Notice
that \(\phi_{\Gamma,\chi}\) is conjugacy invariant but need not be a group
character.  Since \(\kappa_\Gamma\in\{0,\pi\}\), the two factors with
\(\chi\) and \(-\chi\) in \eqref{eq:Rchi-selberg-product-section} are
complex conjugates at the real points \(s=\sigma_m\).
All products below use primitive conjugacy classes, with
\([\gamma]\) and \([\gamma^{-1}]\) counted separately.  The paired
factors in \eqref{eq:Rchi-fredholm-section} agree for these two classes.
Thus indexing that product by unoriented primitive closed geodesics
would require squaring each displayed paired factor.

Set
\begin{equation}\label{eq:sigma-m-definition}
        \sigma_m:=\frac{1+\sqrt{m^2+1}}2,
        \qquad m\geq1,
\end{equation}
and
\begin{equation}\label{eq:q-gamma-m-n}
        q_{\gamma,m,n}
        :=\exp\bigl(-(n+\sigma_m)\ell(\gamma)\bigr).
\end{equation}
The absolute lift-and-character factor is
\begin{equation}\label{eq:Rchi-selberg-product-section}
        \Rcal_{\Gamma,\chi}
        :=
        \prod_{m=1}^{\infty}
        \frac{
        Z_{\Gamma,\chi}^{(m)}(\sigma_m)
        Z_{\Gamma,-\chi}^{(m)}(\sigma_m)
        }{Z_0(\sigma_m)^2}.
\end{equation}
Equivalently,
\begin{equation}\label{eq:Rchi-fredholm-section}
\Rcal_{\Gamma,\chi}
=
\prod_{[\gamma]\in\mathcal P_{\bar\Gamma}}
\prod_{m=1}^{\infty}
\prod_{n=0}^{\infty}
\frac{
1-2q_{\gamma,m,n}\cos\bigl(m\phi_{\Gamma,\chi}(\gamma)\bigr)
+q_{\gamma,m,n}^2
}{(1-q_{\gamma,m,n})^2}.
\end{equation}

\begin{lemma}\label{lem:lift-character-product-convergence}
The products \eqref{eq:character-selberg-product} converge absolutely and
locally uniformly for \(\Re s>1\).  The product
\eqref{eq:Rchi-fredholm-section} converges absolutely, uniformly in
\(\chi\) on the character torus.
\end{lemma}

\begin{proof}
Let
\[
        N_{\mathcal P}(R)
        :=\#\{[\gamma]\in\mathcal P_{\bar\Gamma}:
        \ell(\gamma)\leq R\}.
\]
The coarse prime-geodesic bound
\begin{equation}\label{eq:coarse-prime-geodesic-bound}
        N_{\mathcal P}(R)\leq C\ee^R,
        \qquad R\geq1,
\end{equation}
implies that
\begin{equation}\label{eq:length-exp-sum-convergence}
        \sum_{[\gamma]\in\mathcal P_{\bar\Gamma}}
        \ee^{-\sigma\ell(\gamma)}<\infty,
        \qquad \sigma>1.
\end{equation}
Moreover, compactness gives a positive systole
\begin{equation}\label{eq:systole-positive}
        \ell_0:=\inf_{[\gamma]\in\mathcal P_{\bar\Gamma}}
        \ell(\gamma)>0.
\end{equation}
For \(\Re s=\sigma>1\), the absolute sum of the elementary factors in
\eqref{eq:character-selberg-product} is bounded by
\[
\sum_{[\gamma]}\sum_{n=0}^{\infty}
\ee^{-(\sigma+n)\ell(\gamma)}
\leq
\frac{1}{1-\ee^{-\ell_0}}
\sum_{[\gamma]}\ee^{-\sigma\ell(\gamma)},
\]
which proves the first assertion and non-vanishing in \(\Re s>1\).

For \(0<q<1\) and \(\alpha\in\R\), set
\[
        F(q,\alpha)
        :=\frac{1-2q\cos\alpha+q^2}{(1-q)^2}.
\]
Then
\begin{equation}\label{eq:log-F-expansion}
        \log F(q,\alpha)
        =2\sum_{r=1}^{\infty}
        \frac{1-\cos(r\alpha)}{r}q^r,
\end{equation}
and, since
\(q_{\gamma,m,n}\leq q_0:=\ee^{-\sigma_1\ell_0}<1\),
\begin{equation}\label{eq:log-F-linear-bound}
        0\leq\log F(q_{\gamma,m,n},\alpha)
        \leq\frac{4}{1-q_0}q_{\gamma,m,n}.
\end{equation}
The sequence \(\sigma_m\) grows linearly and satisfies
\(\sigma_1=(1+\sqrt2)/2>1\).  Hence
\[
\sum_{m=1}^{\infty}\ee^{-\sigma_m\ell}
\leq C\ee^{-\sigma_1\ell},
\qquad \ell\geq\ell_0,
\]
and therefore
\[
\sum_{[\gamma]}\sum_{m=1}^{\infty}\sum_{n=0}^{\infty}
q_{\gamma,m,n}
\leq C\sum_{[\gamma]}\ee^{-\sigma_1\ell(\gamma)}<\infty.
\]
Together with \eqref{eq:log-F-linear-bound}, this proves the second
assertion.  The majorant is independent of \(\chi\).
\end{proof}

The logarithm of \(\Rcal_{\Gamma,\chi}\) is consequently the absolutely
convergent series
\begin{equation}\label{eq:log-Rchi-absolute-series}
\begin{aligned}
\log\Rcal_{\Gamma,\chi}
&=2\sum_{m=1}^{\infty}
\sum_{[\gamma]\in\mathcal P_{\bar\Gamma}}
\sum_{r=1}^{\infty}
\frac{1-\cos\bigl(rm\phi_{\Gamma,\chi}(\gamma)\bigr)}{r}
\frac{\ee^{-r\sigma_m\ell(\gamma)}}
{1-\ee^{-r\ell(\gamma)}}.
\end{aligned}
\end{equation}

\begin{proposition}[Closed-geodesic determinant contribution]
\label{prop:geodesic-determinant-contribution}
The Mellin transform of \(\mathcal G_{\Gamma,\chi}\) is regular at
\(s=0\), and
\begin{equation}\label{eq:geodesic-determinant-factor}
        \exp\!\left(
        -\left.\frac{d}{ds}\right|_{s=0}
        \frac1{\Gamma(s)}\int_0^\infty
        t^{s-1}\mathcal G_{\Gamma,\chi}(t)\,dt
        \right)
        =\Rcal_{\Gamma,\chi}.
\end{equation}
\end{proposition}

\begin{proof}
By \eqref{eq:geodesic-correction-heat},
\[
\Theta_{\Gamma,\chi}(t;\gamma^r)-\vartheta(t)
=\sum_{m\in\Z\setminus\{0\}}
\ee^{-m^2t/4}
\left(
\ee^{-\ii mr\phi_{\Gamma,\chi}(\gamma)}-1
\right).
\]
Before interchanging the sums and the heat integral, use
\(|\ee^{-\ii\alpha}-1|\leq2\) and the elementary integral
\[
\int_0^\infty t^{-3/2}
\exp\!\left(-\frac{m^2+1}{4}t
-\frac{r^2\ell^2}{4t}\right)\,dt
=\frac{2\sqrt\pi}{r\ell}
\exp\!\left(-\frac{r\ell}{2}\sqrt{m^2+1}\right).
\]
The resulting absolute majorant is a constant times
\begin{equation}\label{eq:relative-fubini-majorant}
\sum_{m=1}^{\infty}\sum_{[\gamma]}
\sum_{r=1}^{\infty}
\frac{\ee^{-\sigma_mr\ell(\gamma)}}
{r(1-\ee^{-r\ell(\gamma)})},
\end{equation}
Indeed,
\[
\sum_{r=1}^{\infty}
\frac{\ee^{-\sigma_mr\ell}}
{r(1-\ee^{-r\ell})}
\leq
\frac{-\log(1-\ee^{-\sigma_m\ell})}{1-\ee^{-\ell_0}}
\leq C\ee^{-\sigma_m\ell},
\qquad \ell\geq\ell_0,
\]
so \eqref{eq:relative-fubini-majorant} is finite by
\eqref{eq:length-exp-sum-convergence} and the linear growth of
\(\sigma_m\).  The same estimate,
with a small change in the exponent, is locally uniform for \(s\) near
zero.  It justifies the interchange and proves regularity.

Since \(1/\Gamma(s)=s+O(s^2)\), differentiation at zero gives
\begin{align*}
&-\left.\frac{d}{ds}\right|_{s=0}
\frac1{\Gamma(s)}\int_0^\infty
t^{s-1}\mathcal G_{\Gamma,\chi}(t)\,dt\\
&\quad=
-\sum_{m\in\Z\setminus\{0\}}
\sum_{[\gamma]}\sum_{r=1}^{\infty}
\frac{\ee^{-\ii mr\phi_{\Gamma,\chi}(\gamma)}-1}{r}
\frac{\ee^{-\sigma_{|m|}r\ell(\gamma)}}
{1-\ee^{-r\ell(\gamma)}}.
\end{align*}
Pairing \(m\) with \(-m\) yields precisely
\eqref{eq:log-Rchi-absolute-series}, which proves
\eqref{eq:geodesic-determinant-factor}.
\end{proof}

\begin{theorem}
\label{thm:main-twisted-det}
For every character \(\chi:\Gamma\to\R/2\pi\Z\),
\begin{equation}\label{eq:main-twisted-det-fredholm}
\begin{aligned}
\det\nolimits_\zeta'(-\Lcal_{\Gamma,\chi})
&=(4\pi)^2
\exp\bigl(C_{\mathrm{id}}\Vol(M)\bigr)
\exp\!\left(4\pi\zeta_M\!\left(-\frac12\right)\right)\\
&\quad\times
\detF\!\left(I-\ee^{-4\pi\sqrt{\Delta_M'}}\right)^2
\Rcal_{\Gamma,\chi}.
\end{aligned}
\end{equation}
\end{theorem}

\begin{proof}
The decomposition \eqref{eq:untwisted-reduced-split} gives, first in a
common half-plane of convergence and then by continuation,
\[
\zeta_{\Gamma,\chi}(s)
=\zeta_{\vartheta,M}(s)+\zeta_{\mathbb S}(s)
+\Vol(M)\zeta_{\mathrm{id}}(s)
+\frac1{\Gamma(s)}\int_0^\infty
t^{s-1}\mathcal G_{\Gamma,\chi}(t)\,dt.
\]
The first three terms give \(D_{\mathrm{base}}(M)\), and
Proposition~\ref{prop:geodesic-determinant-contribution} gives
\(\Rcal_{\Gamma,\chi}\).
\end{proof}

Taking \(\chi=0\) gives the corrected untwisted formula.
\begin{corollary}\label{thm:untwisted-fredholm-determinant}
\begin{equation}\label{eq:untwisted-fredholm-determinant}
        \det\nolimits_\zeta'(-\Lcal_{\Gamma,0})
        =D_{\mathrm{base}}(M)\Rcal_{\Gamma,0}.
\end{equation}
In particular, the untwisted determinant generally depends on the chosen
lift \(\Gamma<\SL(2,\R)\).
More explicitly, if
\(\varepsilon_\Gamma(\gamma):=\operatorname{sgn}\Tr(\widetilde\gamma)\),
then
\begin{equation}\label{eq:untwisted-lift-factor}
\Rcal_{\Gamma,0}
=\prod_{[\gamma],m,n}
\left(
\frac{1-\varepsilon_\Gamma(\gamma)^m q_{\gamma,m,n}}
{1-q_{\gamma,m,n}}
\right)^2.
\end{equation}
Only negative-trace primitive lifts and odd values of \(m\) contribute
non-trivially to this product.
\end{corollary}

For a fixed lift, the character dependence is a quotient of the absolute
factors.
\begin{proposition}[Relative determinant]
\label{prop:relative-determinant-character}
\begin{equation}\label{eq:relative-determinant-Rchi}
\frac{\det\nolimits_\zeta'(-\Lcal_{\Gamma,\chi})}
{\det\nolimits_\zeta'(-\Lcal_{\Gamma,0})}
=\frac{\Rcal_{\Gamma,\chi}}{\Rcal_{\Gamma,0}}.
\end{equation}
Equivalently,
\begin{equation}\label{eq:relative-product-explicit}
\frac{\Rcal_{\Gamma,\chi}}{\Rcal_{\Gamma,0}}
=\prod_{[\gamma],m,n}
\frac{
1-2q_{\gamma,m,n}\cos\bigl(m\phi_{\Gamma,\chi}(\gamma)\bigr)
+q_{\gamma,m,n}^2
}{
1-2q_{\gamma,m,n}\cos\bigl(m\kappa_\Gamma(\widetilde\gamma)\bigr)
+q_{\gamma,m,n}^2
},
\end{equation}
where the product ranges over
\([\gamma]\in\mathcal P_{\bar\Gamma}\), \(m\geq1\), and \(n\geq0\).
Its logarithm has the equivalent absolutely convergent form
\begin{equation}\label{eq:relative-log-corrected}
\begin{aligned}
\log\frac{\Rcal_{\Gamma,\chi}}{\Rcal_{\Gamma,0}}
&=2\sum_{m=1}^{\infty}\sum_{[\gamma]}
\sum_{r=1}^{\infty}
\frac{
\cos\bigl(rm\kappa_\Gamma(\widetilde\gamma)\bigr)
-\cos\bigl(rm\phi_{\Gamma,\chi}(\gamma)\bigr)
}{r}
\frac{\ee^{-r\sigma_m\ell(\gamma)}}
{1-\ee^{-r\ell(\gamma)}}\\
&=2\sum_{m=1}^{\infty}\sum_{[\gamma]}
\sum_{r=1}^{\infty}
\varepsilon_\Gamma(\gamma)^{mr}
\frac{1-\cos\bigl(mr\chi(\widetilde\gamma)\bigr)}{r}
\frac{\ee^{-r\sigma_m\ell(\gamma)}}
{1-\ee^{-r\ell(\gamma)}},
\end{aligned}
\end{equation}
where \(\varepsilon_\Gamma(\gamma)=\operatorname{sgn}
\Tr(\widetilde\gamma)\).  This real series need not be non-negative.
\end{proposition}

\begin{proof}
Divide Theorem~\ref{thm:main-twisted-det} by its \(\chi=0\) case.  The
absolute convergence proved above justifies taking the quotient term by
term.  Unlike the zero-holonomy-normalized absolute factor, this relative
product need not be greater than or equal to one.
\end{proof}

\begin{remark}[Consistency under a change of lift]
Let \(\epsilon:\bar\Gamma\to\{\pm1\}\) be a sign character and form the
new lift
\[
        \Gamma_\epsilon
        :=\{\epsilon(\gamma)\widetilde\gamma:
        \gamma\in\bar\Gamma\}.
\]
Write \(\eta_\epsilon(\gamma)=0\) or \(\pi\) according as
\(\epsilon(\gamma)=1\) or \(-1\), and define the compensating character
on \(\Gamma_\epsilon\) by
\[
\chi_\epsilon(\epsilon(\gamma)\widetilde\gamma)
:=\chi(\widetilde\gamma)-\eta_\epsilon(\gamma).
\]
Then the actions defining \(P^{\Gamma,\chi}\) and
\(P^{\Gamma_\epsilon,\chi_\epsilon}\) coincide.  Moreover,
\[
\kappa_{\Gamma_\epsilon}(\epsilon(\gamma)\widetilde\gamma)
+\chi_\epsilon(\epsilon(\gamma)\widetilde\gamma)
=\kappa_\Gamma(\widetilde\gamma)+\chi(\widetilde\gamma)
\quad\text{mod }2\pi.
\]
Thus the corrected theta factor, closed-geodesic product, and determinant
are invariant under this simultaneous change.  This is the consistency
that fails if the central phase is omitted.
\end{remark}

\subsection{The universal constant}
\label{subsec:universal-identity-constant}

To conclude the paper we give a formula for the universal constant that enters into the determinant formula.

\begin{proposition}
\label{prop:C-id-explicit}
With \(C_{\mathrm{id}}\) defined by \eqref{eq:C-id-definition}, one has
\begin{align}
C_{\mathrm{id}}
&=
-\frac{\pi}{256}(3+\log 16)                                      \notag\\
&\quad
-\frac{1}{4\pi}
\sum_{k=1}^{\infty}
\sum_{n=1}^{\infty}
(-1)^{n-1}
\frac{n}{k(k+n)}
K_2\!\left(2\pi\sqrt{k(k+n)}\right)                              \notag\\
&\quad
+\frac{1}{\pi}
\sum_{q=1}^{\infty}
\frac{1}{q}
\int_0^\infty
\lambda\tanh(\pi\lambda)
\exp\!\left(-4\pi q\sqrt{\lambda^2+\frac14}\right)
\,d\lambda                                                       \notag\\
&\quad
+\frac{1}{12}
+4\sum_{j=1}^{\infty}
(-1)^{j-1}
\int_0^\infty
\lambda\sqrt{\lambda^2+\frac14}\,
\ee^{-2\pi j\lambda}
\,d\lambda .
\label{eq:C-id-explicit}
\end{align}
Here \(K_\nu\) denotes the modified Bessel function of the second kind.
\end{proposition}

\begin{proof}
By definition,
\[
        C_{\mathrm{id}}
        =
        -\zeta_{\mathrm{id}}'(0),
\]
where
\[
        \zeta_{\mathrm{id}}(s)
        =
        \frac1{\Gamma(s)}
        \int_0^\infty t^{s-1}a_{\mathrm{id}}(t)\,dt
\]
and
\[
        a_{\mathrm{id}}(t)
        =
        \frac{\mathcal I_0(t)}{\Vol(M)}
        -
        \vartheta(t)\frac{\mathcal I_M(t)}{\Vol(M)}.
\]
Thus
\[
        C_{\mathrm{id}}
        =
        -\zeta_{\mathcal I_0}'(0)
        +
        \zeta_{\vartheta\mathcal I_M}'(0),
\]
where
\[
        \zeta_{\mathcal I_0}(s)
        =
        \frac1{\Gamma(s)}
        \int_0^\infty
        t^{s-1}\frac{\mathcal I_0(t)}{\Vol(M)}\,dt
\]
and
\[
        \zeta_{\vartheta\mathcal I_M}(s)
        =
        \frac1{\Gamma(s)}
        \int_0^\infty
        t^{s-1}
        \vartheta(t)\frac{\mathcal I_M(t)}{\Vol(M)}\,dt.
\]

We first compute the contribution of \(\mathcal I_0\).  From
\eqref{eq:identity-contribution-character-sl2},
\[
\frac{\mathcal I_0(t)}{\Vol(M)}
=
\frac{\pi\ee^{-t/4}}{4t^2}
\sum_{k\in\Z}
\frac{
\exp\!\left(-\frac{4k^2\pi^2}{t}\right)
}{
\cosh^2\!\left(\frac{2|k|\pi^2}{t}\right)
}.
\]
The term \(k=0\) gives
\[
\zeta_{\mathcal I_0}^{(0)}(s)
=
\frac{\pi}{4\Gamma(s)}
\int_0^\infty t^{s-3}\ee^{-t/4}\,dt
=
\frac{\pi}{4}\,4^{s-2}
\frac{\Gamma(s-2)}{\Gamma(s)}.
\]
Since
\[
        \frac{\Gamma(s-2)}{\Gamma(s)}
        =
        \frac1{(s-1)(s-2)},
\]
we get
\[
        \left(\zeta_{\mathcal I_0}^{(0)}\right)'(0)
        =
        \frac{\pi}{256}(3+\log 16).
\]

For \(k\neq0\), use
\[
        \frac1{\cosh^2 x}
        =
        4\sum_{n=1}^{\infty}(-1)^{n-1}n\ee^{-2nx},
        \qquad x>0.
\]
After pairing \(k\) and \(-k\), this gives
\[
\zeta_{\mathcal I_0}^{(\neq0)}(s)
=
\frac{2\pi}{\Gamma(s)}
\sum_{k=1}^{\infty}
\sum_{n=1}^{\infty}
(-1)^{n-1}n
\int_0^\infty
 t^{s-3}
\exp\!\left(
-\frac t4
-\frac{4\pi^2k(k+n)}{t}
\right)\,dt.
\]
Using
\[
        \int_0^\infty t^{\nu-1}\ee^{-\beta t-\gamma/t}\,dt
        =
        2\left(\frac{\gamma}{\beta}\right)^{\nu/2}
        K_\nu(2\sqrt{\beta\gamma}),
        \qquad \beta,\gamma>0,
\]
with
\[
        \nu=s-2,
        \qquad
        \beta=\frac14,
        \qquad
        \gamma=4\pi^2k(k+n),
\]
and using \(1/\Gamma(s)=s+O(s^2)\), we obtain
\[
        \left(\zeta_{\mathcal I_0}^{(\neq0)}\right)'(0)
        =
        \frac1{4\pi}
        \sum_{k=1}^{\infty}
        \sum_{n=1}^{\infty}
        (-1)^{n-1}
        \frac{n}{k(k+n)}
        K_2\!\left(2\pi\sqrt{k(k+n)}\right).
\]
Therefore
\begin{equation}\label{eq:zeta-I0-prime-Cid-proof}
\zeta_{\mathcal I_0}'(0)
=
\frac{\pi}{256}(3+\log 16)
+
\frac1{4\pi}
\sum_{k=1}^{\infty}
\sum_{n=1}^{\infty}
(-1)^{n-1}
\frac{n}{k(k+n)}
K_2\!\left(2\pi\sqrt{k(k+n)}\right).
\end{equation}

We next compute the contribution of
\(\vartheta(t)\mathcal I_M(t)\).  By Poisson summation,
\[
        \vartheta(t)
        =
        2\sqrt{\frac{\pi}{t}}
        \sum_{q\in\Z}
        \exp\!\left(-\frac{4\pi^2q^2}{t}\right).
\]
Using \eqref{eq:base-identity-term}, we get
\[
\zeta_{\vartheta\mathcal I_M}(s)
=
\frac1{2\sqrt\pi\,\Gamma(s)}
\sum_{q\in\Z}
\int_{-\infty}^{\infty}
\lambda\tanh(\pi\lambda)
\int_0^\infty
 t^{s-\frac32}
\exp\!\left(
-t\left(\lambda^2+\frac14\right)
-\frac{4\pi^2q^2}{t}
\right)
\,dt\,d\lambda.
\]

For \(q\neq0\), the expression multiplying \(1/\Gamma(s)\) is
holomorphic at \(s=0\).  Since \(1/\Gamma(s)=s+O(s^2)\), differentiating
at zero and using
\[
        \int_0^\infty t^{-3/2}\ee^{-\beta t-\gamma/t}\,dt
        =
        \frac{\sqrt\pi}{\sqrt\gamma}\ee^{-2\sqrt{\beta\gamma}},
\]
we obtain, after pairing \(q\) and \(-q\) and using the evenness of
\(\lambda\tanh(\pi\lambda)\),
\begin{equation}\label{eq:zeta-theta-IM-nonzero-Cid-proof}
\left(\zeta_{\vartheta\mathcal I_M}^{(\neq0)}\right)'(0)
=
\frac1{\pi}
\sum_{q=1}^{\infty}
\frac1q
\int_0^\infty
\lambda\tanh(\pi\lambda)
\exp\!\left(
-4\pi q\sqrt{\lambda^2+\frac14}
\right)
\,d\lambda.
\end{equation}

It remains to compute the \(q=0\) contribution.  It is
\[
\zeta_{\vartheta\mathcal I_M}^{(0)}(s)
=
\frac1{2\sqrt\pi}
\frac{\Gamma\left(s-\frac12\right)}{\Gamma(s)}
\int_{-\infty}^{\infty}
\lambda\tanh(\pi\lambda)
\left(\lambda^2+\frac14\right)^{\frac12-s}
\,d\lambda,
\]
understood by meromorphic continuation.  Since the integrand is even,
\[
\int_{-\infty}^{\infty}
\lambda\tanh(\pi\lambda)
\left(\lambda^2+\frac14\right)^{\frac12-s}
\,d\lambda
=
2\int_0^\infty
\lambda\tanh(\pi\lambda)
\left(\lambda^2+\frac14\right)^{\frac12-s}
\,d\lambda.
\]
Using
\[
        \tanh(\pi\lambda)
        =
        1-\frac{2}{\ee^{2\pi\lambda}+1},
        \qquad \lambda>0,
\]
we split the last integral into a polynomial part and an exponentially
decaying part.

The polynomial part is, by analytic continuation,
\[
2\int_0^\infty
\lambda\left(\lambda^2+\frac14\right)^{\frac12-s}
\,d\lambda
=
\frac{(1/2)^{3-2s}}{s-\frac32}.
\]
Hence
\[
\left.
\frac{d}{ds}
\right|_{s=0}
\left[
\frac1{2\sqrt\pi}
\frac{\Gamma\left(s-\frac12\right)}{\Gamma(s)}
\frac{(1/2)^{3-2s}}{s-\frac32}
\right]
=
\frac1{12}.
\]

For the exponentially decaying part, expand
\[
        \frac1{\ee^{2\pi\lambda}+1}
        =
        \sum_{j=1}^{\infty}(-1)^{j-1}\ee^{-2\pi j\lambda}.
\]
Since
\[
\frac1{2\sqrt\pi}
\frac{\Gamma\left(s-\frac12\right)}{\Gamma(s)}
=
-s+O(s^2),
\qquad s\to0,
\]
this part contributes
\[
4\sum_{j=1}^{\infty}
(-1)^{j-1}
\int_0^\infty
\lambda
\sqrt{\lambda^2+\frac14}\,
\ee^{-2\pi j\lambda}
\,d\lambda.
\]
Equivalently, for \(a,b>0\), one may use the Laplace transform
\[
        \int_0^\infty
        \ee^{-a\lambda}\sqrt{\lambda^2+b^2}\,d\lambda
        =
        \frac{\pi b}{2a}
        \left(\mathbf H_1(ab)-Y_1(ab)\right),
\]
where \(\mathbf H_1\) is the ordinary Struve function and \(Y_1\) is the
Bessel function of the second kind.  Thus
\[
\int_0^\infty
\lambda\sqrt{\lambda^2+b^2}\,
\ee^{-a\lambda}\,d\lambda
=
-\frac{\partial}{\partial a}
\left[
        \frac{\pi b}{2a}
        \left(\mathbf H_1(ab)-Y_1(ab)\right)
\right].
\]
We keep the integral form in \eqref{eq:C-id-explicit}.  Therefore
\begin{equation}\label{eq:zeta-theta-IM-zero-Cid-proof}
\left(\zeta_{\vartheta\mathcal I_M}^{(0)}\right)'(0)
=
\frac1{12}
+
4\sum_{j=1}^{\infty}
(-1)^{j-1}
\int_0^\infty
\lambda\sqrt{\lambda^2+\frac14}\,
\ee^{-2\pi j\lambda}
\,d\lambda.
\end{equation}

Combining
\eqref{eq:zeta-I0-prime-Cid-proof},
\eqref{eq:zeta-theta-IM-nonzero-Cid-proof}, and
\eqref{eq:zeta-theta-IM-zero-Cid-proof}, and using
\[
        C_{\mathrm{id}}
        =
        -\zeta_{\mathcal I_0}'(0)
        +
        \zeta_{\vartheta\mathcal I_M}'(0),
\]
gives \eqref{eq:C-id-explicit}.
\end{proof}

\begin{remark}
   Numerically, one obtains
\[
C_{\mathrm{id}} \simeq 0.06536605482497744.
\]
\end{remark}


\begin{thebibliography}{99}

\bibitem{BauerFurutani}
W.~Bauer and K.~Furutani,
\emph{Spectral zeta function of a sub-Laplacian on product
sub-Riemannian manifolds and zeta-regularized determinant},
J. Geom. Phys. \textbf{60} (2010), 1209--1234.

\bibitem{BauerFurutaniIwasaki}
W.~Bauer, K.~Furutani and C.~Iwasaki,
\emph{Spectral zeta function of the sub-Laplacian on two step
nilmanifolds},
J. Math. Pures Appl. \textbf{97} (2012), 242--261.

\bibitem{BauerFurutaniIwasakiPseudoH}
W.~Bauer, K.~Furutani and C.~Iwasaki,
\emph{Spectral zeta function on pseudo H-type nilmanifolds},
Indian J. Pure Appl. Math. \textbf{46} (2015), 539--582.

\bibitem{BaudoinDemni}
F.~Baudoin and N.~Demni,
\emph{Integral representation of the sub-elliptic heat kernel on the
complex anti-de Sitter fibration},
Arch. Math. \textbf{111} (2018), 399--406.

\bibitem{Bonnefont}
M.~Bonnefont,
\emph{The subelliptic heat kernel on $\SL(2,\mathbb R)$ and on its
universal covering: integral representations and some functional
inequalities},
Potential Anal. \textbf{36} (2012), 275--300.

\bibitem{ChoiTakhtajan}
C.~Choi and L.~A.~Takhtajan,
\emph{Supersymmetry and trace formulas II. Selberg trace formula},
arXiv:2306.13636v2 (2025).

\bibitem{DHP}
E.~D'Hoker and D.~H.~Phong,
On determinants of Laplacians on Riemann surfaces,
\emph{Communications in Mathematical Physics} \textbf{104} (1986),
537--545.

\bibitem{Hejhal}
D.~A. Hejhal,
\emph{The Selberg Trace Formula for $PSL(2,\mathbb R)$, Vols. I--II},
Lecture Notes in Mathematics, Vols. 548 and 1001, Springer, 1976 and 1983.

\bibitem{IkedaMatsumoto}
N.~Ikeda and H.~Matsumoto,
\emph{Brownian motion on the hyperbolic plane and Selberg trace formula},
J. Funct. Anal. \textbf{163} (1999), 63--110.

\bibitem{McKean1972}
H.~P. McKean,
Selberg's trace formula as applied to a compact Riemann surface,
\emph{Comm. Pure Appl. Math.} \textbf{25} (1972), 225--246.
MR0473166.

\bibitem{McKean1974}
H.~P. McKean,
Correction to: ``Selberg's trace formula as applied to a compact Riemann surface''
(Comm. Pure Appl. Math. \textbf{25} (1972), 225--246),
\emph{Comm. Pure Appl. Math.} \textbf{27} (1974), 134.
MR0473167.

\bibitem{OsgoodPhillipsSarnak1988}
B.~Osgood, R.~Phillips and P.~Sarnak,
\newblock Extremals of determinants of Laplacians,
\newblock {\em Journal of Functional Analysis} {\bf 80} (1988), no.~1,
148--211.

\bibitem{Patterson}
S.~J. Patterson,
\emph{The Selberg zeta-function of a Kleinian group},
in Number Theory, Trace Formulas and Discrete Groups,
Academic Press, 1989.

\bibitem{PongeHeisenberg}
R.~Ponge,
\emph{Heisenberg calculus and spectral theory of hypoelliptic operators
on Heisenberg manifolds},
Mem. Amer. Math. Soc. \textbf{194} (2008), no.~906.

\bibitem{PongeResidue}
R.~Ponge,
\emph{Noncommutative residue for Heisenberg manifolds. Applications in
CR and contact geometry},
J. Funct. Anal. \textbf{252} (2007), 399--463.

\bibitem{RaySinger1971}
D.~B. Ray and I.~M. Singer,
\newblock R-torsion and the Laplacian on Riemannian manifolds,
\newblock {\em Advances in Mathematics} {\bf 7} (1971), 145--210.

\bibitem{RuminSeshadri}
M.~Rumin and N.~Seshadri,
\emph{Analytic torsions on contact manifolds},
Ann. Inst. Fourier \textbf{62} (2012), 727--782.

\bibitem{Sarnak1987}
P.~Sarnak,
\newblock Determinants of Laplacians,
\newblock {\em Communications in Mathematical Physics} {\bf 110} (1987),
113--120.

\bibitem{Wang}
J.~Wang,
\emph{The subelliptic heat kernel on the anti-de Sitter space},
Potential Anal. \textbf{45} (2016), 635--653.
\end{thebibliography}
\end{document}